\documentclass[12pt, a4paper, reqno]{amsart}
\usepackage[utf8]{inputenc}
\usepackage[T1]{fontenc}
\usepackage{bbm}

\usepackage[dvipsnames]{xcolor}
\usepackage[polish, english]{babel}

\usepackage{makecell}
\usepackage{slashbox}
\usepackage{wrapfig}

\usepackage{hyperref}
\usepackage{graphicx}
\usepackage{subfig}

\usepackage{marginnote}

\usepackage{dsfont}
\usepackage{mathrsfs}
\usepackage{pstricks}

\usepackage[foot]{amsaddr}
\usepackage{orcidlink}

\newtheorem{theorem}{Theorem}

\newtheorem{corollary}[theorem]{Corollary}
\newtheorem{lemma}[theorem]{Lemma}

\newtheorem{proposition}[theorem]{Proposition}

\theoremstyle{remark}
\newtheorem{remark}{Remark}

\numberwithin{equation}{section}
\numberwithin{theorem}{section}
\numberwithin{remark}{section}

\definecolor{ks}{rgb}{0.7,0.1,0.2}

\def \R {\mathbb{R}}
\def \Rd {\mathbb{R}^d}
\def \Rdz {\mathbb{R}^d_0}

\def \NN {\mathbb{N}}

\renewcommand{\d}{{\rm d}}

\usepackage{scalerel,stackengine}
\stackMath
\newcommand\reallywidehat[1]{%
\savestack{\tmpbox}{\stretchto{%
  \scaleto{%
    \scalerel*[\widthof{\ensuremath{#1}}]{\kern-.6pt\bigwedge\kern-.6pt}%
    {\rule[-\textheight/2]{1ex}{\textheight}}
  }{\textheight}%
}{0.5ex}}%
\stackon[1pt]{#1}{\tmpbox}%
}

\author[T. Jakubowski]{Tomasz Jakubowski\,\orcidlink{0000-0001-6349-7593}}
\email{tomasz.jakubowski@pwr.edu.pl}

\author[K. Kaleta]{Kamil Kaleta\,\orcidlink{0000-0002-1766-524X}}
\email{kamil.kaleta@pwr.edu.pl}

\author[K. Szczypkowski]{Karol Szczypkowski*\, \orcidlink{0000-0001-9986-0788}}
\address{
	Wydzia{\l} Matematyki,
	Politechnika Wroc{\l}awska\\
	Wyb. Wyspia\'{n}skiego 27\\
	50-370 Wroc{\l}aw\\
	Poland
}
\email{karol.szczypkowski@pwr.edu.pl}
\thanks{* Corresponding author.}

\title{Relativistic stable operators  with critical potentials}

\begin{document}

\begin{abstract}
We give local in time sharp two sided estimates 
of the heat kernel associated with the 
relativistic stable operator perturbed by a critical (Hardy) potential.
\end{abstract}

\maketitle



\vspace{\baselineskip}

\section{Introduction}\label{sec:main}
Let $d\in\NN:=\{1,2,\ldots\}$ and $\alpha\in (0,2\land d)$. For $\delta \in (0,d - \alpha)$
we consider the 
operator 
\begin{align}\label{def:operator}
L=-(-\Delta+1)^{\alpha/2}+V_{\delta}(x)\,,
\end{align}
where 
\begin{align}\label{def:V}
V_{\delta}(x)=
\frac{2^{\alpha/2}\Gamma\left(\frac{d-\delta}{2}\right)}{\Gamma\left(\frac{d-\delta-\alpha}{2}\right)}  \frac{K_{\frac{\delta+\alpha}{2}}\left(|x|\right)}{K_{\frac{\delta}{2}}\left(|x|\right) }
|x|^{-\alpha/2}\,.
\end{align}
$\Gamma$ denotes the Gamma function and $K_{\nu}$ is the modified Bessel function of the second kind. 
The potential $V_\delta$ arises naturally as a critical (Hardy) potential for the relativistic operator when implementing the approach developed in Bogdan et al.\ \cite{MR3460023}.
For $\delta=\frac{d-\alpha}{2}$ it was derived and investigated  
by Roncal \cite{MR4294647}
from three different perspectives.

The main purpose of the present paper is to analyse
the heat kernel corresponding to $L$, that is,
the fundamental solution 
to the parabolic equation
$\partial_t u = L u$.
Analogous problems were studied for the classical Hardy operator $\Delta + \kappa|x|^{-2}$ and its fractional counterpart $-(-\Delta)^{\alpha/2} + \kappa|x|^{-\alpha}$ with $\alpha \in (0,2)$, see Subsection~\ref{subsec:Hist} for related literature. Contrary to the latter two, the operator $-(-\Delta+1)^{\alpha/2} + V_\delta(x)$ is not homogeneous; the kinetic term $-((-\Delta+1)^{\alpha/2}-1)$ manifests properties of both operators: $\Delta$ and $-(-\Delta)^{\alpha/2}$. This generates new challenges and requires the development of an appropriate methodology. The methods we propose allow us to treat in a common framework other important operators, e.g., the relativistic operator with Coulomb potential 
\begin{align*}
-\sqrt{-\Delta+1}+\kappa |x|^{-1}.
\end{align*}

\subsection{Main results}

In our first result we address the question of the existence of the heat kernel corresponding to
the operator $L$.  
We use the notation $\Rdz=\Rd\setminus \{0\}$.
\begin{proposition}\label{prop:weak_sol}
Let $\delta\in (0,d-\alpha)$.
There is a Borel function $\tilde{p}\colon (0,\infty)\times \Rdz\times \Rdz \to (0,\infty)$
such that for every 
$s\in\R$, $x\in\Rdz$ and
$\phi \in C_c^{\infty}(\R\times\Rd)$
we have
$$
\int_s^{\infty}\int_{\Rd} \tilde{p}(u-s,x,z)
\left[\partial_u - (-\Delta_z+1)^{\alpha/2}+V_\delta(z) \right] \phi(u,z) \,\d{z}\d{u}
 =-\phi(s,x)\,.
$$
The double integral is absolutely convergent.
\end{proposition}

The function $\tilde{p}$ is constructed in Subsection~\ref{sec:Sch-P} by the use of the perturbation technique of kernels.
Our main result is Theorem~\ref{thm:MAIN}, which is concerned with the estimates of $\tilde{p}$. For $t>0$, $x\in\Rdz$ we define

\begin{align*}
H^0(t,x)&=
1+t^{\delta/\alpha}|x|^{-\delta} \qquad \quad &&\mbox{if } \quad \delta\in (0,\tfrac{d-\alpha}{2}]\,,\\
H^0(t,x)&=1+t^{(d-\alpha-\delta)/\alpha}|x|^{-(d-\alpha-\delta)}
\qquad \quad &&\mbox{if } \quad \delta\in [\tfrac{d-\alpha}{2},d-\alpha)\,.
\end{align*}
Let $p$ be the heat kernel for the operator $-(-\Delta +1)^{\alpha/2}$ (see \eqref{def:p-m} with $m=1$).
Here, and in what follows, we write $f\approx g$
on $D$ if $f,g \geq 0$ and there is a (comparability) constant $c\geq 1$ such that
$c^{-1} g \leq f \leq c g$ holds on $D$.

\begin{theorem}\label{thm:MAIN}
Let $\alpha\in (0,2\land d)$ and $\delta \in (0, d-\alpha)$. 
The function $\tilde{p}$ is 
jointly continuous on $(0,\infty)\times\Rdz \times \Rdz$
and for every $T>0$ we have
$$
\tilde{p}(t,x,y)\approx p(t,x,y) H^0(t,x)H^0(t,y)
$$
on $(0,T]\times \Rdz \times \Rdz$.
The comparability constant can be chosen to depend only on $d,\alpha, \delta, T$.
\end{theorem}

\noindent
Note that local in time sharp estimates of $p$ are known.
Namely, for every $T>0$ we have 
$$
p(t,x,y)\approx \left( t^{-d/\alpha} \land 
\frac{t K_{\frac{d+\alpha}{2}}(|x-y|)}{|x-y|^{\frac{d+\alpha}{2}}}\right)
$$
on $(0,T]\times\Rd\times\Rd$,
see Subsection~\ref{sec:rel-den}. We first prove Theorem \ref{thm:MAIN} for $\delta \in (0,\tfrac{d-\alpha}{2}]$ in Theorem~\ref{thm:main-1}, and then in Theorem~\ref{thm:all_comp} we obtain an extended version of Theorem \ref{thm:main-1}, which in particular covers the case $\delta \in [\tfrac{d-\alpha}{2},d-\alpha)$, but also potentials of the form $\kappa |x|^{-\alpha}$.

We give another property that relates the operator $-(-\Delta +1)^{\alpha/2}$ and the potential $V_\delta$, and formally provides the representation of the quadratic form of $L=-(-\Delta +1)^{\alpha/2}+V_\delta$.
For $f\in L^2(\Rd)$ we let 
$$
P_t f(x)=\int_{\Rd} p(t,x,y)f(y)\,\d{y}\qquad \mbox{and}\qquad \mathcal{E}(f,f)=\lim_{t\to 0^+} \frac1t \left<f-P_t f,f\right>\,.
$$
Since $P_t$ is a strongly continuous contraction semigroup on $L^2(\Rd)$, then $\mathcal{E}$ is the  quadratic form corresponding to $-(-\Delta+1)^{\alpha/2}$, see \cite[Lemma~1.3.4]{MR2778606}.
As already mentioned, the potential $V_\delta$ is constructed by using the approach developed in Bogdan et al.\ \cite{MR3460023},
therefore as a consequence we get the following Hardy identity (see \cite[Theorem 2]{MR3460023}.
For $\delta=\frac{d-\alpha}{2}$ this result was obtained by Roncal \cite{MR4294647} via different methods.

\begin{corollary}\label{thm:3}
Let $\delta\in (0,d-\alpha)$.
Then for every $f\in L^2(\Rd)$,
\begin{align*}
\mathcal{E}(f,f)= \int_{\Rd} f^2(x)V_\delta(x)\,\d{x} 
+\frac12 \int_{\Rd} \int_{\Rd} \left[
\frac{f(x)}{h_\delta(x)}-\frac{f(y)}{h_\delta(y)}\right]^2 h_\delta(y)h_\delta(x) \nu(x-y)\,\d{y}\d{x}\,,
\end{align*}
where
$$
h_{\delta}(x)= c_{d,\alpha,\delta} \,|x|^{-\frac{\delta}{2}} K_{\frac{\delta}{2}}(|x|)\,,
\qquad
\nu(x)= \frac{2^{\frac{\alpha-d}{2}+1}}{\pi^{d/2}|\Gamma(-\frac{\alpha}{2})| }\, \frac{K_{\frac{d+\alpha}{2}}(|x|)}{|x|^{\frac{d+\alpha}{2}}}\,.
$$
\end{corollary}

\noindent
The exact value of the constant $c_{d,\alpha,\delta}>0$ is
 given in  Lemma~\ref{lem:h-expl},
but it does not matter in Corollary~\ref{thm:3} due to cancellations taking place in the double integral expression.

The role of the function $h_\delta$ and its properties are important also in other places of our reasoning. For instance,
in Theorem \ref{thm:tph=h}, we prove that for each $\delta \in (0, (d-\alpha)/2]$, for all $t>0$ and $x\in \Rdz$,
\begin{align*}
\int_{\R^d} \tilde{p}(t,x,y) h_{\delta}(y)\,\d{y}
=h_{\delta}(x)\,.
\end{align*}
The latter is crucial in the proof of the lower bound in Subsection~\ref{subsec:lower_bound}.
In light of that equality, the identity in Corollary \ref{thm:3} may be interpreted as the ground state representation for the quadratic form of the operator $L$.
\begin{remark}\label{rem:1}
The above results can be extended to 
heat kernels $\tilde{p}^m$ corresponding to
operators
$-(-\Delta+m^{2/\alpha})^{\alpha/2}+V_{\delta}^m(x)$, $m>0$, by proper scaling as explained in Lemmas~\ref{lem:scal-m} and~\ref{lem:scal-m-tp}.
In particular, 
for every $\delta\in(0,d-\alpha)$ and $T>0$ we have
$$
\tilde{p}^m(t,x,y)\approx p^m(t,x,y) H^0(t,x)H^0(t,y)
$$
on $(0,T/m]\times \Rdz \times \Rdz$
with  comparability constant
independent of $m>0$ 
(in fact, that depends only on $d,\alpha, \delta, T$).
The heat kernel $p^m$ is defined in \eqref{def:p-m}, see also Lemma~\ref{lem:scal-m}.
Additionally, by taking $m\to 0^+$ we can recover the estimates of $\tilde{p}^{\,0}(t,x,y)$ from
\cite[Theorem~1.1]{MR3933622}, see Proposition~\ref{prop:m-to-zero}.
\end{remark}

\subsection{Methods}

As mentioned in Remark~\ref{rem:1}, the sharp two-sided heat kernel estimates for $-(-\Delta)^{\alpha/2}+\kappa_\delta |x|^{-\alpha}$ were obtained in \cite{MR3933622}. The potential $\kappa_\delta |x|^{-\alpha}$ is the critical perturbation of the fractional Laplacian, that is, the value of $\kappa_\delta$ significantly affects the behaviour of the corresponding heat kernel. In particular, if $\kappa_\delta > \kappa^*$ (certain explicit constant) the heat kernel becomes instantaneously infinite. 

The methods used in \cite{MR3933622} cannot be transferred to our problem
and, for this reason, we have to find a new strategy.
Below, we outline main steps that led to the sharp two-sided estimates in Theorem~\ref{thm:MAIN}. 
We start with the upper estimates:

\begin{enumerate}
	\item[$(U_1)$] upper heat kernel estimates for $-(-\Delta)^{\alpha/2}+V_\delta(x)$ \, (Proposition \ref{prop:upper_bound}),
\item[$(U_2)$] upper heat kernel estimates for $-(-\Delta+1)^{\alpha/2}+V_\delta(x)$\, (Theorem \ref{thm:tp-upper}).
\end{enumerate}

\noindent
We take the heat kernel estimates from  \cite{MR3933622} as the starting point, and after a careful analysis of the potential $V_\delta(x)$, in particular of the difference $V_\delta(x)-\kappa_\delta |x|^{-\alpha}$ (see Lemma \ref{lem:diff_est}),
by applying perturbation technique we arrive at $(U_1)$.
The latter  is again non-trivial, not only because $V_\delta(x)-\kappa_\delta |x|^{-\alpha}>0$ ($V_\delta(x)$ is larger than $\kappa_\delta |x|^{-\alpha}$), but also since the difference is unbounded (singular) in the neighbourhood of zero if $\delta<\alpha$, while the heat kernel of $-(-\Delta)^{\alpha/2}+\kappa_\delta |x|^{-\alpha}$ is already  singular at zero.
The standard realisation of the perturbation procedure fails, because the kernel of this operator does not satisfy the 3G inequality typically used in the first  step of the method. 
When proceeding towards $(U_2)$, one encounters certain technical difficulties. For instance, the heat kernel $p$ of $-(-\Delta+1)^{\alpha/2}$ does not have scaling, and has exponential decay in the spatial
variable, which is harder to retain compared to the power-type nature of the heat kernel of the fractional Laplacian, which has scaling. 
We develop a new integral method to take into account that faster decay, and combine it with $(U_1)$, to finally obtain the proper outcome in $(U_2)$. From that we also get:

\begin{enumerate}
	\item[$(U_2^*)$] upper heat kernel estimates for $-(-\Delta+1)^{\alpha/2}+\kappa_\delta |x|^{-\alpha}$\,  (Theorem \ref{thm:all_comp}).
\end{enumerate}

\noindent 
Since sharp bounds of $p$
are only known locally in time, this is reflected in our results.

Here are the key steps for the lower estimates:

\begin{enumerate}

\item[$(L_1)$] invariance of $h_\delta$ \, (Theorem \ref{thm:tph=h}),

\item[$(L_2)$] lower heat kernel estimates for $-(-\Delta+1)^{\alpha/2}+V_\delta(x)$ \, (Theorem \ref{thm:tp-lower}).
\end{enumerate}

\noindent
First of all, the formulae for $h_\delta$ and $V_\delta$
are not accidental, see Subsection~\ref{subsec:derivation}. They are computed according to a general procedure proposed in \cite{MR3460023}, which we used specifically to the operator $-(-\Delta+1)^{\alpha/2}$. From \cite{MR3460023} it is known that $h_\delta$ is the
so-called super-median. In $(L_1)$ we prove more, namely that $h_\delta$ is actually invariant for the relativistic semigroup perturbed by $V_\delta(x)$. As far as $(L_1)$ is intuitively expected, the proof is technical, especially if $\delta = \frac{d-\alpha}{2}$  (in view of blow-up result this value of $\delta$ may be regarded as critical).
The step $(L_1)$ is inevitable in order to carry out an exact integral analysis and prove $(L_2)$.
Let us point out that in $(L_1)$, as in the whole Section~\ref{sec:our}, we require $\delta \in (0,\tfrac{d-\alpha}{2}]$. To prove $(L_2)$ for $\delta \in (\tfrac{d-\alpha}{2}, d-\alpha)$ additional work is needed, see Remark~\ref{rem:reflect_beta}.

We comment on the steps for the lower heat kernel estimates of the relativistic operator with Coulomb potential:

\begin{enumerate}
\item [$(L_1^*)$] two-sided heat kernel estimates of $-(-\Delta+1)^{\alpha/2}+V_\delta(x)$,
\item[$(L_2^*)$] lower heat kernel estimates of $-(-\Delta+1)^{\alpha/2}+\kappa_\delta |x|^{-\alpha}$. 
\end{enumerate}

\noindent
We note that the potential $V_\delta(x)$  is tailor-made for $-(-\Delta+1)^{\alpha/2}$, which manifests  in $(L_1)$ and  $(L_2)$. Clearly, $(L_1^*)$ is a consequence of $(U_2)$ and $(L_2)$.
In order to prove $(L_2^*)$,
we perturb the heat kernel of $-(-\Delta+1)^{\alpha/2}+V_\delta(x)$
by (possibly singular) potential $\kappa_\delta |x|^{-\alpha}-V_\delta(x)<0$.
To succeed, we heavily rely on $(L_1^*)$ and adapt the new integral method used to prove $(U_2)$.

In Section~\ref{sec:further} we summarize all the estimates and
 explain the blow-up phenomenon, that is, the criticality of the parameter $\delta^*:=\frac{d-\alpha}{2}$.

\subsection{Historical and bibliographical comments}\label{subsec:Hist}
The analysis of heat kernels corresponding to operators with the so-called critical potentials goes back to Baras and Goldstein
\cite{MR742415}, where the existence 
of non-trivial non-negative solutions of the  heat equation $\partial_t u=\Delta u +\kappa |x|^{-2}u$
in $\Rd$
was  established for $0\leq \kappa\leq (d-2)^2/4$,
and non-existence (explosion) for bigger constants $\kappa$.
The operator was also studied by Vazquez  and Zuazua
\cite{MR1760280} in bounded subsets of $\Rd$ as well as in the whole space. Sharp estimates of the heat kernel were obtained by Liskevich and Sobol
\cite{MR1953267} for $0<\kappa<(d-2)^2/4$, and by
Milman and Semenov \cite[Theorem~1]{MR2064932} for
$0<\kappa\leq (d-2)^2/4$, see also \cite{MR2114706}.
Sharp estimates in bounded domains were given by
Moschini and Tesei \cite{MR2328115}
in the subcritical case, and by
Filippas et al.\ \cite{MR2308757} for the critical value $(d-2)^2/4$.
Further generalizations for local operators were given in a series of papers by Metafune et al. that include
\cite{MR3630331, MR4375654}.
Asymptotics of solutions to the Cauchy
problem by self-similar solutions were
proved by Pilarczyk  \cite{MR3020137}.

Another operator that drew attention in this context was the fractional Laplacian 
with Hardy potential $-(-\Delta)^{\alpha/2}+\kappa |x|^{-\alpha}$.
Abdellaoui et al. \cite{MR3479207, MR3492734}
proved that if $\kappa>\kappa^*=2^{\alpha}\Gamma\left(\frac{d+\alpha}{4}\right)/\Gamma\left(\frac{d-\alpha}{4}\right)$, then the operator has no weak positive supersolution, while for $0<\kappa\leq \kappa^*$ non-trivial non-negative solutions exist.
In the latter case 
Bogdan et al.\ \cite{MR3933622} obtained
sharp two-sided estimates
of the heat kernel, namely that it is comparable on $(0,\infty)\times\Rdz\times\Rdz$ to the expression
$$p^0(t,x,y) H^0(t,x)H^0(t,y)\,.$$
Here $p^0$ is the heat kernel of the fractional Laplacian and satisfies on $(0,\infty)\times\Rd\times\Rd$,
$$p^{0}(t,x,y) \approx \left(t^{-d/\alpha} \land \frac{t}{ |x-y|^{d+\alpha}}\right).$$
The estimates from \cite{MR3933622} were a key ingredient in the analysis of Sobolev norms
by Frank et al.\ \cite{MR4206613}, Merz \cite{MR4311597},
and Bui and D'Ancona \cite{AB-PDA-2022}.
They were also used by
Bui and Bui \cite{MR4319351} 
to study 
maximal regularity of the parabolic equation,
and by
Bhakta et al.\ \cite{MR4104934}
to represent weak solutions.
Hardy spaces of the operator were investigated 
by Bui and Nader in \cite{MR4423543}.
Bogdan et al.\ \cite{KB-TJ-PK-DP-2022}
found asymptotics of the heat kernel by studying self-similar solutions.
Cholewa et al.\ \cite{MR4372781}
studied the parabolic equation (also of order greater than 2) in the context of homogeneity.
We also refer to
BenAmor \cite{MR4216547},
 Chen and Weth
\cite{MR4315592},
Jakubowski and Maciocha \cite{TJ-PM-2022}
for the fractional Laplacian with Hardy potential on subsets of $\Rd$,
and to Frank et al.\ \cite{MR2425175},
where the operator $|x|^{-\beta}(-(-\Delta)^{\alpha/2}+\kappa |x|^{-\alpha}-1)$, for certain $\beta>0$, was treated.
Perturbations of the fractional Laplacian,
and more general operators, by negative critical potentials were considered by
Jakubowski and Wang \cite{MR4140086}, Cho et al.\ \cite{MR4163128}, Song et al.~\cite{RS-PW-SW-2022}.

The relativistic  operator 
$\sqrt{-\Delta+m^2}$ is an important object in physical studies, because it
describes the kinetic
energy of a relativistic particle with mass~$m$.
In quantum mechanics it was used in
problems concerning the stability of relativistic matter, in particular, the relativistic operator with Coulomb potential $\sqrt{-\Delta+m^2}-\kappa |x|^{-1}$ is of interest,
see Weder \cite{MR351332, MR0402547}, Herbst \cite{MR436854}, Daubechies and Lieb
\cite{MR763750, MR719430}, 
Fefferman and Llave \cite{MR864658}, Carmona et al. \cite{MR1054115},
Frank et al. \cite{MR2335782, MR4196148}, Lieb and Seiringer \cite{MR2583992}.
We note in passing that in Subsection~\ref{subsec:4td} we provide local in time sharp estimates of the heat kernel 
corresponding to that operator
and in \cite{TJ-KK-KS-2022} we prove 
pointwise estimates of its eigenfunctions.

The relativistic stable operators $- ((-\Delta + m^{2/\alpha})^{\alpha/2}-m)$, $\alpha\in (0,2)$, $m>0$, were investigated by Ryznar \cite{MR1906405}, who obtained Green function and Poisson kernel estimates on bounded domains as well as Harnack inequality.
Kulczycki and Siudeja \cite{MR2231884}
studied intrinsic ultracontractivity of the associated Feynman-Kac semigroup.
After these two papers the topic was intensely studied and resulted in rich literature concerning such operators and corresponding stochastic processes,
see e.g. 
\cite{MR2316878,
MR2386098,
MR2923420, MR2917772,
MR2925182,
MR4116699,
MR4206411,
MR4100844,
MR4381659}.
The undertaken topics 
involve also linear or non-linear mostly elliptic equations or systems of equations 
with or without critical potentials, and with certain focus on the unique continuation properties, see for instance
Fall and Felli \cite{MR3237776, MR3393257},
Secchi \cite{MR4020755},
Ambrosio \cite{MR4343747},
Bueno et al.\ \cite{MR4406449}
and the references therein.
The list is far from being complete.
Results more closely related to the present paper 
can be found in Grzywny et al.\ \cite{TG-KK-PS-2022}, where
perturbations of non-local operators by a proper Kato class potentials are considered, and 
include relativistic stable operators.

\subsection{Notation and organization}

We use $:=$ to indicate the definition.
As usual $a \land b := \min\{a,b\}$, $a \vee b := \max\{a,b\}$. For a function $f(x)$ which is radial, i.e. its value depends only on $r=|x|$, we use the same letter to denote its profile $f(r):=f(x)$.
We write $c=c(a,\ldots)$ to indicate that the constant $c$ depends only on the listed parameters.
We also recall that $\Rdz=\Rd\setminus \{0\}$.
In certain parts of the presented theory the functions, series or integrals are allowed to attain the infinite value. On the other hand, we often avoid it by restricting the domain to $\Rdz$. It is though sometimes replaced by $\Rd$, for instance in the integration regions, since one point is of the Lebesgue measure zero.
For $n\in\NN$ we denote by $C_0(\R^n)$ the space of continuous functions $f\colon \R^n\to \R$ that vanish at infinity, and $C_c^{\infty}(\R^n)$
are smooth functions with 
compact support.

The paper is organized as follows.
The preliminary Section~\ref{sec:Prel} is divided into four parts.
First in Subsection~\ref{sec:Sch-P} we introduce the general framework of Schr{\"o}dinger perturbations of transition densities, which is used in the paper. 
In Subsection~\ref{subsec:rel-op} we provide the context for the relativistic stable operator.
Next, in Subsection~\ref{subsec:derivation} we present computations that give rise to $V_\delta$ and $h_\delta$ and we prove Corollary~\ref{thm:3}. Finally, in Subsection~\ref{subsec:p-analysis} we study properties of the potential $V_\delta$.
Section~\ref{sec:largest}
is mainly devoted to the analysis of the heat kernel corresponding to the fractional Laplacian perturbed by $V_\delta$. In the same section we prove Proposition~\ref{prop:weak_sol}.
In Section~\ref{sec:our}, which consists of four parts, we focus on the heat kernel
corresponding to the relativistic stable operator perturbed by $V_\delta$ for $\delta\in (0,\tfrac{d-\alpha}{2}]$.
In subsequent subsections we show upper bounds, 
invariance of $h_\delta$ with respect to perturbed semigroup and lower bounds.
In Subsection~\ref{subsec:thm1-proof} we
prove Theorem~\ref{thm:main-1}.
In Section~\ref{sec:further}
we extend Theorem~\ref{thm:main-1}
to other transition densities and potentials, and we discuss blow-up phenomenon.
In Appendix~\ref{sec:app}
we collect known properties of the modified Bessel function of the second kind and of the heat kernel 
and the L{\'e}vy measure
corresponding to the relativistic stable operator.

\section{Preliminaries}\label{sec:Prel}

\subsection{Schr{\"o}dinger perturbation}\label{sec:Sch-P}

The following subsection is general and independent of more specific framework of Section~\ref{sec:main}.
Let
$p\colon (0,\infty)\times\Rdz \times\Rdz \to [0,\infty]$
be a Borel function satisfying for all $0<s<t$ and $x,y\in\Rdz$,
$$
\int_{\R^d} p(s,x,z) p(t-s,z,y)\,\d{z} = p(t,x,y)\,.
$$
We call $p$ a transition density.
For a Borel function $q \colon \Rdz \to [0,\infty]$
we define the Schr{\"o}dinger perturbation of 
$p$ by $q$ as
$$
\tilde{p}_q =\sum_{n=0}^{\infty}p_n\,,
$$
where, for $t>0$, $x,y\in\Rdz$, we let $p_0(t,x,y)=p(t,x,y)$ and
$$
p_n(t,x,y)=\int_0^t\int_{\R^d}p_{n-1}(s,x,z) q(z)p(t-s,z,y)\,\d{z}\d{s}\,,\qquad n\geq 1\,.
$$
From the general theory 
developed in \cite{MR2457489},
based solely on the algebraic structure of the above series and the Fubini-Tonelli theorem, the Duhamel's formula,
\begin{align}\label{eq:Duh}
\tilde{p}_q(t,x,y)=p(t,x,y)+\int_0^t\int_{\R^d} \tilde{p}_q(s,x,z)q(z)p(t-s,z,y)\,\d{z}\d{s}\,,
\end{align}
and the Chapman-Kolmogorov equation,
\begin{align}\label{eq:C-K}
\int_{\R^d} \tilde{p}_q(s,x,z) \tilde{p}_q(t-s,z,y)\,\d{z} = \tilde{p}_q(t,x,y)\,,
\end{align}
hold.
In particular, $\tilde{p}$ is a transition density.
Furthermore, for every $q_1,q_2 \geq 0$ we have
\begin{align}\label{eq:pert-of-pert}
\tilde{p}_{q_1+q_2} = \widetilde{(\tilde{p}_{q_1})}_{q_2}\,.
\end{align}
Namely, the perturbation of $p$ by $q_1+q_2$ may be realized in two steps: first by obtaining $\tilde{p}_{q_1}$, and then by  perturbing $\tilde{p}_{q_1}$ by $q_2$.
Suppose that $\rho^1$ and $\rho^2$ are two transition densities
such that $\rho^1 \leq c \rho^2$
on $(0,T]\times\Rdz \times\Rdz$ for some $T>0$ and $c\geq 0$. Then on $(0,T]\times\Rdz \times\Rdz $ we have
\begin{align}\label{ineq:tran_den}
\tilde{\rho}^1_q \leq c\, \tilde{\rho}^2_{c\cdot q}\,.
\end{align}

We also consider (similarly to above) 
 perturbations by singed $q$. 
In that case we have to make sure that the series converges properly and that it is non-negative.
For convenience we merge arguments used in
\cite{MR2457489} in such a way that fits well our setting and applications.
\begin{lemma}\label{lem:Kato-signed}
Suppose that $q_1\geq 0$ and $\tilde{p}_{q_1}(t,x,y)<\infty$ for every $t>0$, $x,y \in \Rdz$. Assume that there are  $\epsilon\in [0,1/2)$ \textup{(}for $q_2\leq 0$ we only require that $\epsilon\in[0,1)$\textup{)} and $\tau>0$ such that for all $t\in (0,\tau]$, $x,y\in \Rdz$,
$$
\int_0^t\int_{\Rd} \tilde{p}_{q_1}(s,x,z)|q_2(z)|\tilde{p}_{q_1}(t-s,z,y)\,\d{z}\d{s}\leq \epsilon \, \tilde{p}_{q_1}(t,x,y)\,.
$$
Then
$\tilde{p}_{q_1+q_2} = \widetilde{(\tilde{p}_{q_1})}_{q_2}$
is a finite transition density. Furthermore, for every $T>0$ there exists a constant
$c=c(\epsilon, \tau, T)>0$
such that 
$\tilde{p}_{q_1+q_2} \geq c \tilde{p}_{q_1}$
on $(0,T]\times \Rdz\times \Rdz$.
\end{lemma}
\begin{proof}
Clearly,
$\tilde{p}_{q_1}$
is a finite transition density, therefore
by \cite[Theorem~2]{MR2457489}
the series 
$\widetilde{(\tilde{p}_{q_1})}_{q_2}$
is a transition density.
By \eqref{eq:pert-of-pert} and \cite[Theorem~2]{MR2457489}
we have
$\tilde{p}_{q_1+|q_2|} = \widetilde{(\tilde{p}_{q_1})}_{|q_2|}\leq c \tilde{p}_{q_1}<\infty$
on $(0,T]\times\Rdz\times\Rdz$ for any $T>0$.
This guarantees the absolute convergence of the series 
and
by \cite[Lemma~8]{MR2457489}
gives $\tilde{p}_{q_1+q_2} = \widetilde{(\tilde{p}_{q_1})}_{q_2}$.
Now, for $q_2\leq 0$, like in \cite[(25)]{MR2457489}
we have
$$
\widetilde{(\tilde{p}_{q_1})}_{q_2}
\geq (1-\epsilon) \,\tilde{p}_{q_1} \,,
$$
and for general (signed) $q_2$ we get $$\widetilde{(\tilde{p}_{q_1})}_{q_2}
\geq (1 - \epsilon/(1-\epsilon))\,\tilde{p}_{q_1}\,,
$$
on $(0,\tau]\times\Rdz\times\Rdz$, which extends to $(0,T]\times \Rdz\times \Rdz$ by the Chapman-Kolmogorov equation. 
\end{proof}

\subsection{The relativistic stable operator}\label{subsec:rel-op}
We briefly recall fundamental properties of the
relativistic stable operator $- ((-\Delta + m^{2/\alpha})^{\alpha/2}-m)$, where $\alpha\in(0,2)$, $m\geq 0$. In fact, we focus on the operator
\begin{align}\label{def:op-m}
-(-\Delta+m^{2/\alpha})^{\alpha/2} f =  \mathcal{F}^{-1} (-\psi^m \hat{f})\,, \qquad f\in C_c^{\infty}(\Rd)\,,
\end{align}
where
\begin{align*}
\psi^m(\xi)=(|\xi|^2+m^{2/\alpha})^{\alpha/2}\,.
\end{align*}
The function $-\psi^m$ is called the symbol or the Fourier multiplier of the operator \eqref{def:op-m}.
The value $\psi^m(0)=m$ is known as the killing rate.
We refer the reader to
\cite{MR3156646},
\cite{MR1739520},
\cite{MR2978140}
for a broader perspective and details
of the material presented below.
It is well known
that the operator \eqref{def:op-m} uniquely generates a translation invariant
Feller semigroup
$(P_t^m)_{t\geq 0}$
or (equivalently) 
vaguely continuous convolution semigroup of measures $p_t^m(dx)$
or (equivalently) a L{\'e}vy process $(X_t^m)_{t\geq 0}$ with exponential killing rate $m$.
For $f\in C_0(\Rd)$ we have that
$$P_t^m f(x)=\int_{\Rd} f(x+z)p_t^m(dz)=\mathbb{E}f(x+X_t^m) \qquad \mbox{and}\qquad
\widehat{p^m_t}(\xi)=e^{-t\psi^m(\xi)}\,.$$
Due to the latter equality,
$\psi^m$ is also referred to as the characteristic exponent and admits the L{\'e}vy-Khintchine representation
$\psi^m(\xi)=m+\int_{\Rd\setminus\{0\}} \big( 1-\cos\left<\xi,z\right>\!\big) \nu^m(z)\d{z}
$, where
\begin{align}\label{def:nu_def}
\nu^m(x)= \frac{2^{\frac{\alpha-d}{2}+1}}{\pi^{d/2}|\Gamma(-\frac{\alpha}{2})|}\, \frac{m^{\frac{d+\alpha}{\alpha}} K_{\frac{d+\alpha}{2}}(m^{\frac1\alpha}|x|)}{(m^{\frac1\alpha}|x|)^{\frac{d+\alpha}{2}}}\,, \quad m>0\,,
\end{align}
and 
$$
\nu^0(x)= \frac{ 2^{\alpha}\Gamma((d+\alpha)/2)}{\pi^{d/2}|\Gamma(-\alpha/2)|} |x|^{-d-\alpha}$$
may be obtained by taking the limit as $m\to 0^+$, see \eqref{eq:lim_K}.
The measure $\nu^m(x)\d{x}$ is called the L{\'e}vy measure.
Since $\widehat{p^m_t}$ is integrable we have
$p^m_t(dx)=p_t^m(x)\d{x}$
and the heat kernel $p^m(t,x,y)=p_t^m(y-x)$ 
corresponding to the operator \eqref{def:op-m}
may be recovered from the symbol by using
the inverse Fourier transform
\begin{align}\label{def:p-m}
p^m(t,x,y)=(2\pi)^{-d}\int_{\Rd} e^{-i \left<y-x,z\right>}\, e^{-t\, \psi^m(z)} \,\d{z}\,.
\end{align}
It will be convenient for us to use an alternative 
equivalent approach to the
semigruop $(P_t^m)_{t\geq 0}$ (or the process $(X_t^m)_{t\geq 0}$)
by the subordination technique.
Before we move further, we note that 
$(P_t^m)_{t\geq 0}$ is 
a strongly continuous contraction semigroup on $C_0(\Rd)$, its infinitesimal generator  
has $C_c^{\infty}(\Rd)$ as a core and for $f\in C_c^{\infty}(\Rd)$ it coincides  with
the operator \eqref{def:op-m} which
is a non-local integro-differential operator 
\begin{align*}
-(-\Delta+m^{2/\alpha})^{\alpha/2} f(x)=-mf(x)+
P.V. 
\int_{\Rd\setminus \{0\}} \big( f(x+z)-f(x)\big) \nu^m(z)\d{z}\,.
\end{align*}
Let $\eta_t(s)$ be the probability density of the $\alpha/2$-stable subordinator
and  
$$\Pi(\d{s})= \frac{\alpha/2}{\Gamma(1-\alpha/2)}\frac{\mathds{1}_{s>0}}{s^{1+\alpha/2}}\d{s}$$
be the corresponding L{\'e}vy measure.
The Laplace transform of
$e^{-m^{2/\alpha}s}\eta_t(s)$ is equal to
$$
\int_0^{\infty} e^{-\lambda s}e^{-m^{2/\alpha}s} \eta_t(s)\,\d{s}=
e^{-t\phi^m(\lambda)}\,,
$$
where the Laplace exponent $\phi^m$
is a Bernstein function and
admits the following
L{\'e}vy-Khintchine representation
$$\phi^m(\lambda)=(\lambda+m^{2/\alpha})^{\alpha/2}=
m+\int_0^{\infty}\left(1-e^{-\lambda s}\right)e^{-m^{2/\alpha}s}\, \Pi(\d{s})\,.
$$
Let $g_t(x)=(4\pi t)^{-d/2}\exp\{-|x|^2/(4t)\}$  be the Gauss-Weierstrass kernel.
The Bochner subordination of the Gaussian semigroup
with respect to $\phi^m$
results in the following relations 
$$
\psi^m(\xi)=\phi^m(|\xi|^2)\,,
\qquad \quad
\nu^m(x)=\int_0^{\infty} g_s(x) e^{-m^{2/\alpha}s}\,\Pi(\d{s})\,,
$$
and
\begin{align}\label{def:p^m}
p^m(t,x,y)=\int_0^{\infty} g_s(x-y) e^{-m^{2/\alpha} s} \eta_t(s)\,\d{s}\,,
\qquad t>0,\, x,y\in\R^d\,.
\end{align}
Note that it is a sub-probabilistic kernel, namely $\int_{\Rd} p^m(t,x,y)\d{y}=e^{-mt}$, $t>0$.
Clearly, it is also a finite transition density
as defined in Subsection~\ref{sec:Sch-P}.
In Subsection~\ref{sec:rel-den} we collect further important properties of $\nu^1(x)$ and $p^1(t,x,y)$.

\subsection{Derivation of $V_\delta$ and $h_\delta$}\label{subsec:derivation}

As already announced in the introduction we use the approach proposed in \cite{MR3460023}.
We assume that $\alpha\in (0,2\land d)$ and $m\geq 0$.
For $x\in\Rdz$ and $\beta \in (0,d)$
we let
$$
h^m_{\beta}(x):=\int_0^{\infty}t^{\frac{d-\alpha-\beta}{\alpha}} p^m(t,x,0)\,\d{t}\,,
$$
and for $\beta \in (-\alpha,d-\alpha)$,
$$
\bar{h}^m_{\beta}(x):=\int_0^{\infty} \left( t^{\frac{d-\alpha-\beta}{\alpha}}\right)'\! p^m(t,x,0)\,\d{t}=  \frac{d-\alpha-\beta}{\alpha} h^m_{\beta+\alpha}(x)\,.
$$
Finally, for $\beta \in (0,d-\alpha)$ we define
$$
V_{\beta}^m(x):= \frac{\bar{h}^m_{\beta}(x)}{h^m_{\beta}(x)}
=  \frac{d-\alpha-\beta}{\alpha} \frac{ h^m_{\beta+\alpha}(x)}{ h^m_{\beta}(x)}\,.
$$
\begin{lemma}\label{lem:scal-m}
For $\beta\in(0,d-\alpha)$, $m>0$ and $x,y\in\Rdz$ we have for 
$\bar{f}(x)=f(x/m^{1/\alpha})$, $f\in C_c^{\infty}(\Rd)$,
$$\nu^m(x)= m^{\frac{d+\alpha}{\alpha}} \nu^1(m^{\frac1\alpha}x)\,,
\qquad (-\Delta+m^{2/\alpha})^{\alpha/2} f(x)
=m\left((-\Delta+1)^{\alpha/2} \bar{f}\right)(m^{\frac1\alpha}x)\,,
$$
and
$$
p^m(t,x,y)= m^{\frac{d}{\alpha}} p^1(mt,m^{\frac{1}{\alpha}} x, m^{\frac{1}{\alpha}}y)\,, \qquad h^m_{\beta}(x)= m^{\frac{\beta}{\alpha}} h^1_{\beta}(m^{\frac{1}{\alpha}} x)\,,
\quad V^m_{\beta}(x)=m V^1_{\beta}(m^{\frac{1}{\alpha}}x)\,.
$$
\end{lemma}
\begin{proof}
The first equality is trivial.
The second one follows from the first and representation of the operator as a non-local integro-differential operator.
Recall that $\eta_{st}(u)=s^{-2/\alpha} \eta_t(s^{-2/\alpha}u)$.
Together with \eqref{def:p^m} this leads to the second line of equalities above.
\end{proof}

In view of Lemma~\ref{lem:scal-m}, in what follows we fix $m=1$ and we remove it from the notation, 
$$
\boxed{
\, \nu(x)=\nu^1(x)\,,\qquad p(t,x,y)=p^1(t,x,y)\,, \qquad h_{\beta}(x)=h^1_{\beta}(x)\,,
\qquad V_{\beta}(x)=V^1_{\beta}(x)\,.
\,
}
$$

\begin{lemma}\label{lem:h-expl}
For $\beta\in (0,d)$, $x\in\Rdz$ we have
$$
h_{\beta}(x)=\frac{2^{\frac{\beta}{2}+1}\Gamma\left(\frac{d-\beta}{\alpha}\right)}{(4\pi)^{d/2} \Gamma\left(\frac{d-\beta}{2}\right)} \,|x|^{-\frac{\beta}{2}} K_{\frac{\beta}{2}}(|x|)\,.
$$
The function $h_\beta(r)$ is decreasing in $r>0$.
\end{lemma}
\begin{proof}
We have
$\int_0^{\infty} t^{(d-\alpha-\beta)/\alpha} \eta_t(s)\,\d{t}=\frac{\Gamma\left(\frac{d-\beta}{\alpha}\right)}{\Gamma\left(\frac{d-\beta}{2}\right)}s^{\frac{d-\beta}{2}-1}$ by \cite[(23)]{MR3460023}.
Thus
\begin{align*}
h_{\beta}(x)
&=\int_0^{\infty}  \int_0^{\infty} g_s(x) e^{-s} t^{(d-\alpha-\beta)/\alpha} \eta_t(s)\,\d{s}\d{t}
=\frac{\Gamma\left(\frac{d-\beta}{\alpha}\right)}{\Gamma\left(\frac{d-\beta}{2}\right)}
\int_0^{\infty} g_s(x) e^{-s}
s^{\frac{d-\beta}{2}-1} \,\d{s}\\
&= \frac{\Gamma\left(\frac{d-\beta}{\alpha}\right)}{(4\pi)^{d/2}\Gamma\left(\frac{d-\beta}{2}\right)}
\int_0^{\infty} e^{-s-\frac{|x|^2}{4s}}
s^{-\frac{\beta}{2}-1} \,\d{s}\,.
\end{align*}
The final formula follows from the definition of $K_{\nu}$, see Subsection~\ref{sec:ap-Bessel}.
Since $K_{\beta/2}(r)$ is decreasing so is $h_\beta(r)$, see Subsection~\ref{sec:ap-Bessel}.
\end{proof}

\begin{remark}\label{rem:VisV}
It is now clear that $V_\delta=\bar{h}_\delta/h_\delta$ as constructed above  coincides with~\eqref{def:V}.
\end{remark}

Recall that
$\tilde{p}$ is the 
Schr{\"o}dinger perturbation of
$p$ by $V_\delta$ for $\delta\in (0,d-\alpha)$ according  to Subsection~\ref{sec:Sch-P}. 
We show how to recover the case of $m>0$ 
from that with $m=1$.
\begin{lemma}\label{lem:scal-m-tp}
Let $\tilde{p}^{\,m}$ be
the 
Schr{\"o}dinger perturbation of
$p^m$ by $V^m_\delta$,
and $(p^m)_n$ be the summand of the corresponding series. Then
for all $t>0$, $x,y\in\Rdz$, $n=0,1,\ldots$
and $m>0$,
$$
(p^m)_n(t,x,y)= m^{\frac{d}{\alpha}}\, p_n (mt,m^{\frac{1}{\alpha}} x, m^{\frac{1}{\alpha}}y)\,,
$$
and
$$
\tilde{p}^{\,m}(t,x,y)= m^{\frac{d}{\alpha}}\, \tilde{p} (mt,m^{\frac{1}{\alpha}} x, m^{\frac{1}{\alpha}}y)\,.
$$
\end{lemma}
\begin{proof}
By Lemma~\ref{lem:scal-m} the first equality holds for $n=0$.
Then by induction
\begin{align*}
(p^m)_{n+1}(t,x,y)
&=\int_0^t\int_{\Rd}
(p^m)_n(s,x,z) V^m_\delta(z) p^m(t-s,z,y)\,\d{z}\d{s}\\
&=m^{\frac{2d}{\alpha}+1}\int_0^t\int_{\Rd}
p_n(ms,m^{\frac{1}{\alpha}}x,m^{\frac{1}{\alpha}}z) V_\delta(m^{\frac{1}{\alpha}}z) p(mt-ms,m^{\frac{1}{\alpha}}z,m^{\frac{1}{\alpha}}y)\,\d{z}\d{s}\\
&= m^{\frac{d}{\alpha}} \int_0^{mt}\int_{\Rd} p_n(u,m^{\frac{1}{\alpha}}x,w)V_\delta (w) p(mt-u,w,m^{\frac{1}{\alpha}}y)\,\d{z}\d{s}\\
&= m^{\frac{d}{\alpha}} p_{n+1}(mt,m^{\frac{1}{\alpha}}x,m^{\frac{1}{\alpha}}y)\,.
\end{align*}
Summing over $n=0,1\ldots$ gives the second equality.
\end{proof}

From \cite[Theorem~1]{MR3460023} we immediately obtain the following inequality.
\begin{corollary}\label{cor:BDK}
For $\delta\in (0,d-\alpha)$ and all $t>0$, $x\in\Rdz$,
$$
\int_{\R^d}\tilde{p}(t,x,y) h_\delta(y)\, \d{y}\leq 
h_\delta(x)\,.
$$
\end{corollary}

\begin{proof}[Proof of Corollary~\ref{thm:3}]
Due to a direct application of
\cite[Theorem~2, formula~(16)]{MR3460023}
it suffices to verify that
\begin{align}
\lim_{t\to 0^+}\int_{\Rd} \int_{\Rd}
\frac{p(t,x,y)}{2t}
&\left[
\frac{f(x)}{h_\delta(x)}-\frac{f(y)}{h_\delta(y)}\right]^2 h_\delta(y)h_\delta(x)\,\d{y}\d{x} \nonumber \\
 =
\frac12 \int_{\Rd} \int_{\Rd} &\left[
\frac{f(x)}{h_\delta(x)}-\frac{f(y)}{h_\delta(y)}\right]^2 h_\delta(y)h_\delta(x) \nu(x-y)\,\d{y}\d{x}\,.\label{eq:H-aux}
\end{align}
Since by \eqref{approx:p} and Lemma~\ref{lem:toLm} we have
$p(t,x,y)/t \leq c \nu(x-y)$ and $p(t,x,y)/t \to \nu(x-y)$
as $t\to 0^+$,
the equality \eqref{eq:H-aux}
holds by the dominated convergence theorem if the right hand side of \eqref{eq:H-aux} is finite, and by Fatou's lemma in the opposite case.
\end{proof}

\subsection{Analysis of the potential}\label{subsec:p-analysis}

We assume that $\alpha\in (0,2\land d)$. 
In view of Lemmas~\ref{lem:scal-m} and~\ref{lem:h-expl}, and Remark~\ref{rem:VisV} we have for $\beta\in(0,d-\alpha)$, $m>0$ and $x\in\Rdz$,
\begin{align}\label{eq:V}
V_{\beta}^m(x)=\kappa_\beta |x|^{-\alpha}\, \frac{\Gamma\left(\frac{\beta}{2}\right)}{\Gamma\left(\frac{\beta+\alpha}{2}\right)}  \frac{K_{\frac{\beta+\alpha}{2}}\left(m^{1/\alpha}|x|\right)}{K_{\frac{\beta}{2}}\left(m^{1/\alpha}|x|\right) } \left(\frac{m^{1/\alpha}|x|}{2}\right)^{\alpha/2},
\end{align}
where
$$
\kappa_\beta = \frac{2^{\alpha} \Gamma\left(\frac{\beta+\alpha}{2}\right)\Gamma\left(\frac{d-\beta}{2}\right)}{\Gamma\left(\frac{\beta}{2}\right)\Gamma\left(\frac{d-\beta-\alpha}{2}\right)}\,.
$$
Note that $\kappa_\beta |x|^{-\alpha}$ is the Hardy potential for the fractional Laplacian.
We write
$$\nu_1=\beta/2\,, \qquad\qquad \nu_2=(\beta+\alpha)/2\,.$$

The following 
two properties stem from
\cite[Lemma~2.6]{MR486686} and
\cite[Theorem~2.9]{MR3637943}, respectively,
\begin{align}\label{properties}
\frac{V_\beta (x)}{\kappa_\beta |x|^{-\alpha}}
\quad \mbox{is radial increasing in } x, \mbox{ and decreasing in } \beta\in(0,d-\alpha).
\end{align}
We proceed with the analysis of $V_\beta$.
\begin{lemma}\label{lem:lim_V}
For $\beta\in(0,d-\alpha)$ we have
$$
\lim_{|x|\to 0} \frac{V_\beta (x)}{\kappa_\beta|x|^{-\alpha}}=1\,.
$$
\end{lemma}
\begin{proof}
The result follows from the asymptotic behaviour of the function $K_{\nu}(r)$ given in \eqref{eq:lim_K}.
\end{proof}

\begin{corollary}\label{cor:diff_nonneg}
For $\beta\in(0,d-\alpha)$ and $x\in\Rdz$ we have $V_\beta(x)- \kappa_\beta |x|^{-\alpha}> 0
$.
\end{corollary}
\begin{proof}
The inequality follows from 
Lemma~\ref{lem:lim_V} and \eqref{properties}.
\end{proof}

Since $V_\beta$ is radial we write $V_\beta(r)=V_\beta(x)$ for $r=|x|$.

\begin{lemma}\label{lem:V-decreasing}
For $\beta\in(0,d-\alpha)$ and $r>0$ we have $V_\beta '(r)<0$.
\end{lemma}
\begin{proof}
Note that $V_\beta(r)= c r^{-\alpha/2} \frac{K_{\nu_2}(r)}{K_{\nu_1}(r)}$. Therefore,
\begin{align}\label{eq:V_prime}
V_\beta'(r)=-(\alpha/2) r^{-1} V_\beta(r)+ 
 V_\beta(r) \left( 
\frac{K_{\nu_2}'(r)}{K_{\nu_2}(r)}
-
\frac{K_{\nu_1}'(r)}{K_{\nu_1}(r)}
\right).
\end{align}
Using \eqref{def_K-1} we get
$$
V_\beta'(r) = V_\beta(r) \left(\frac{K_{\nu_1+1}(r)}{K_{\nu_1}(r)} -\frac{K_{\nu_2+1}(r)}{K_{\nu_2}(r)}\right).
$$
The result follows from \eqref{K/K-1}.
\end{proof}

\begin{lemma}\label{lem:diff}
For $\beta\in(0,d-\alpha)$ and $x\in\Rdz$ we have
$$
V_\beta(x)- \kappa_\beta |x|^{-\alpha} =\kappa_\beta
|x|^{-\alpha} \int_0^{|x|} \frac{V_\beta(s)}{\kappa_\beta s^{-\alpha}} 
\left(
\frac{K_{\nu_1-1}(s)}{K_{\nu_1}(s)}
-\frac{K_{\nu_2-1}(s)}{K_{\nu_2}(s)} \right) \d{s} \,.
$$
\end{lemma}
\begin{proof}
By Lemma~\ref{lem:lim_V},
$$
V_\beta(x)- \kappa_\beta |x|^{-\alpha}
=|x|^{-\alpha} \int_0^{|x|} \left(s^{\alpha} V_\beta(s) \right)' \d{s}\,,
$$
where 
$(s^{\alpha} V_\beta(s))'=\alpha s^{\alpha -1} V_\beta(s) + s^{\alpha}V_\beta'(s)$. 
We use \eqref{eq:V_prime}
and  \eqref{def_K-2} to get
$$
V_\beta'(s)=-\alpha s^{-1} V_\beta(s)+ 
 V_\beta(s)
\left(
\frac{K_{\nu_1-1}(s)}{K_{\nu_1}(s)}
-\frac{K_{\nu_2-1}(s)}{K_{\nu_2}(s)} \right).
$$
This ends the proof.
\end{proof}

\begin{lemma}\label{lem:diff_est}
For $\beta\in(0,d-\alpha)$ we have on $\{x\in\Rd \colon 0<|x|\leq 1/2\}$,

\begin{align*}
V_\beta(x)- \kappa_\beta |x|^{-\alpha} &\approx
\begin{cases}
 |x|^{2-\alpha}\,, &\beta> 2\,,\\
 |x|^{2-\alpha}\log(1/|x|)\,, \qquad   &\beta =2\,,\\
 |x|^{\beta-\alpha}\,,  &\beta \in (0,2)\,.
\end{cases}
\end{align*}
\end{lemma}
\begin{proof} By the asymptotics given in
Subsection~\ref{sec:ap-Bessel}
we have
\begin{align*}
\lim_{r\to 0^+} \frac{K_{\nu-1}(r)}{K_\nu(r)} (r/2)^{-1}=\frac1{\nu-1}\,,\qquad &\nu>1\,,\\
\lim_{r\to 0^+} \frac{K_{\nu-1}(r)}{K_\nu(r)}\, (r \log(1/r))^{-1}=1\,,\qquad &\nu=1\,,\\
\lim_{r\to 0^+} \frac{K_{\nu-1}(r)}{K_\nu(r)}\, (r^{-1+2\nu})^{-1} = 2^{1-2\nu} \frac{\Gamma(1-\nu)}{\Gamma(\nu)}\,, \qquad &\nu \in (0,1)\,.
\end{align*}
Therefore, from the two expressions
$K_{\nu_1-1}(r)/K_{\nu_1}(r)$
and $K_{\nu_2-1}(r)/K_{\nu_2}(r)$
the one with $\nu_1$ dominates the other when $r\to 0^+$.
Then by
Lemmas~\ref{lem:lim_V} and~\ref{lem:diff},
and using the above limits with $\nu=\beta/2$ we obtain the comparability.
\end{proof}

As shown in \cite[Proof of Proposition~5]{MR3460023} the function $\beta \mapsto \kappa_\beta$ is increasing on $(0,(d-\alpha)/2]$ and decreasing on $[(d-\alpha)/2,d-\alpha)$. The mapping $\beta \mapsto V_\beta(x)$ does not have that property.
Combining \cite{MR3460023} and \eqref{properties} we get the following observation, which we also depict on Figure~\ref{fig}.
\begin{remark}\label{rem:reflect_beta}
For $\beta \in (\tfrac{d-\alpha}{2},d-\alpha)$
and $\beta'=d-\alpha-\beta \in(0,\frac{d-\alpha}{2})$
we have
$$
\kappa_\beta=\kappa_{\beta'} \qquad \mbox{and}
\qquad
V_{\beta} < V_{\beta'}\,.
$$
\end{remark}

\begin{figure}[h]
    \centering
    \subfloat[\centering $|x|=0.1$]{{\includegraphics[width=0.45\linewidth]{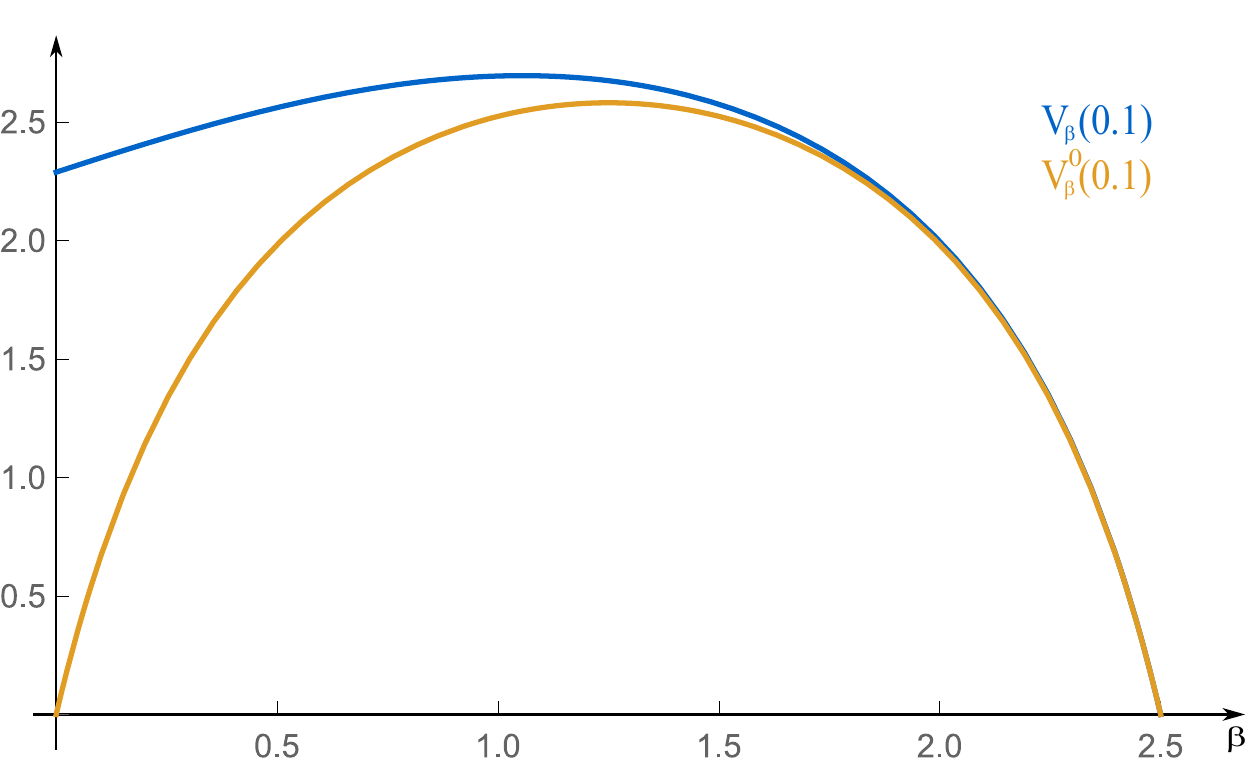} }}
    \qquad
    \subfloat[\centering $|x|=1$]{{\includegraphics[width=0.45\linewidth]{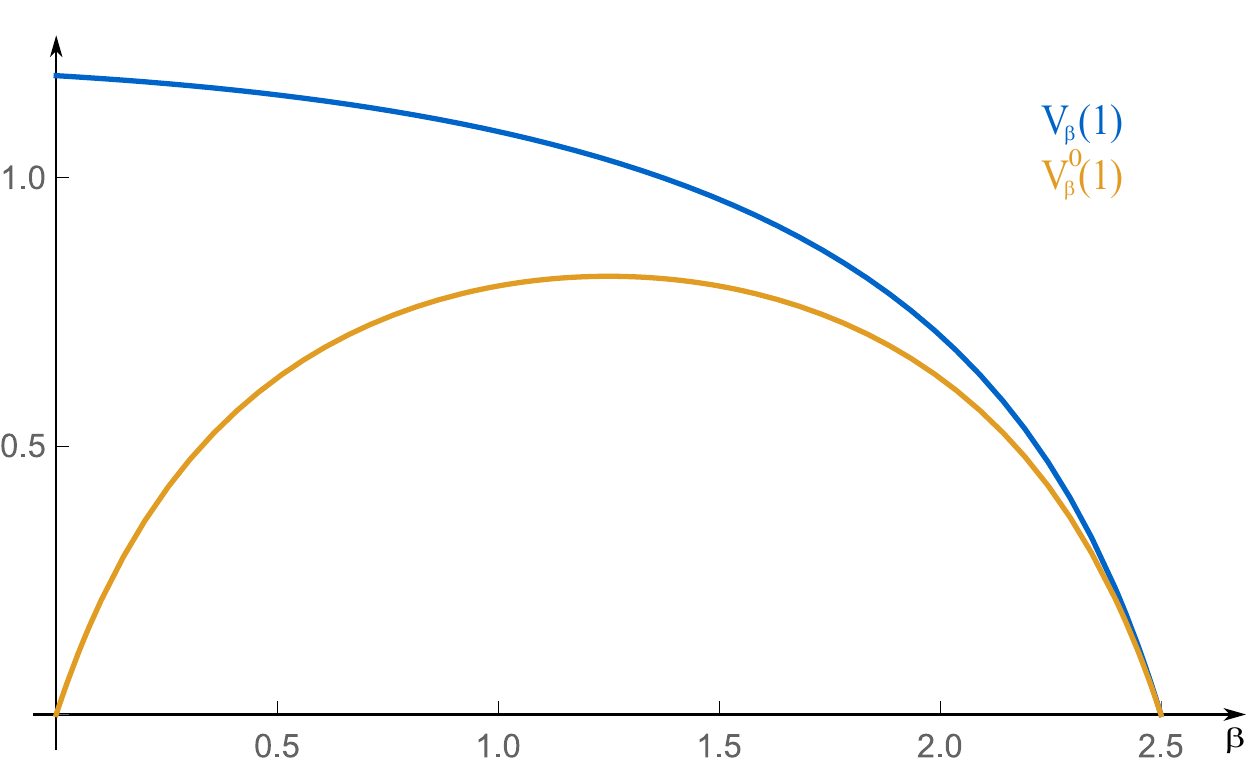} }}%
    \caption{Comparison of $V_\beta(x)$ and $V_\beta^0(x)=\kappa_\beta |x|^{-\alpha}$ for $d=3$, $\alpha=1/2$.}
    \label{fig}
\end{figure}

We will later on need the following technical result that provides the upper bound of the difference
$V_\beta(x)- \kappa_\beta |x|^{-\alpha}$, which is uniform in the parameter $\beta$. 
\begin{lemma}\label{lem:diff_unif_bound}
Let $0<\beta_0< \alpha\land (d-\alpha)$.
There exists $c>0$ such that
for all $\beta\in [\beta_0,d-\alpha)$ and $0<|x| \leq 1$,
$$
V_\beta(x)- \kappa_\beta |x|^{-\alpha} 
\leq c |x|^{\beta_0-\alpha}\,.
$$
\end{lemma}
\begin{proof}
By Lemma~\ref{lem:diff}, \eqref{properties} and \eqref{K/K-2},
$$
V_\beta(x)- \kappa_\beta |x|^{-\alpha} 
\leq \kappa_{\frac{d-\alpha}{2}}
|x|^{-\alpha}  \frac{V_{\beta_0} (1)}{\kappa_{\beta_0}} 
\int_0^{|x|}
\frac{K_{\frac{\beta_0}{2}-1}(s)}{K_{\frac{\beta_0}{2}}(s)}\, \d{s} \,.
$$
Then we use the third limit from the proof of Lemma~\ref{lem:diff_est}.
\end{proof}

\section{Heat kernel of $-(-\Delta)^{\alpha/2}+V_\delta$}\label{sec:largest}

In this section we assume that $\alpha\in (0,2\land d)$.
Before treating the heat kernel for
the operator \eqref{def:operator},
we first analyse 
the one corresponding to
$-(-\Delta)^{\alpha/2}+V_\delta$.
Namely, we consider the heat kernel
$p^{0}$ of $-(-\Delta)^{\alpha/2}$, see \eqref{def:p^m} for the definition, and we concentrate on
$\tilde{p}^{\,0}_{V_\delta}$
that is the Schr{\"o}dinger
perturbation of $p^{0}$ by $V_\delta$.
As an auxiliary function  we use
$\tilde{p}^{\,0}_{V^0_\delta}$ that is
the  Schr{\"o}dinger perturbation of $p^{0}$ by $V^0_\delta(z)=\kappa_\delta |z|^{-\alpha}$, which was thoroughly investigated in~\cite{MR3933622}.
By Corollary~\ref{cor:diff_nonneg} we have for $\delta\in(0,d-\alpha)$ that
$$
0 \leq  V^0_\delta \leq V_\delta  \qquad \mbox{and}\qquad \tilde{p}^{\,0}_{V_\delta^0} \leq \tilde{p}^{\,0}_{V_{\delta}}  \,.
$$
In Proposition~\ref{prop:upper_bound} below
we show that the converse of the latter inequality
holds up to multiplicative constant.
We first prove two auxiliary results.

\begin{lemma}\label{lem:int-aux}
Let $\beta \in (0,d)$ and $\gamma \in [\beta-\alpha,\infty)$. There is a constant $c>0$ such that  for all $t>0$, $x\in\Rd$,
\begin{align*}
\int_0^t \int_{\Rd} s^{\gamma/\alpha} p^{0}(s,x,z) |z|^{-\beta} \,\d{z}\d{s} \le 
c \begin{cases}
 t^{(\gamma+\alpha)/\alpha}\left(t^{-\beta/\alpha} \land |x|^{-\beta}\right)
 \quad &{\rm for}\,\, \gamma> \beta-\alpha \,, \\ 
\\
\log (1+t^{\beta/\alpha}|x|^{-\beta}) &{\rm for}\,\, \gamma=\beta-\alpha\,.
\end{cases}
\end{align*}
\end{lemma}
\begin{proof}
By \cite[Lemma~2.3]{MR3933622},
\begin{align*}
\int_{\Rd} p^{0}(s,x,z) |z|^{-\beta} \d{z} \le c(s^{-\beta/\alpha} \land |x|^{-\beta})\,.
\end{align*}
Hence, for $|x|^\alpha \ge t$,
\begin{align*}
\int_0^t\int_{\Rd} s^{\gamma/\alpha} p^{0}(s,x,z) |z|^{-\beta} \d{z}\d{s} &\le c|x|^{-\beta} \int_0^t s^{\gamma/\alpha}\d{s} =  \frac{c\alpha}{\gamma+\alpha}t^{(\gamma+\alpha)/\alpha}|x|^{-\beta}\,.
\end{align*}
Let $|x|^\alpha < t$. For $\gamma>\beta-\alpha$,
\begin{align*}
\int_0^t\int_{\Rd} s^{\gamma/\alpha} p^{0}(s,x,z) |z|^{-\beta} \d{z}\d{s} &\le c|x|^{-\beta} \int_0^{|x|^\alpha} s^{\gamma/\alpha}\d{s} + c\int_{|x|^\alpha}^t s^{(\gamma-\beta)/\alpha} \d{s}  \le c t^{(\gamma+\alpha - \beta)/\alpha}\,.
\end{align*}
The same steps lead to the second estimate for $\gamma=\beta-\alpha$.
\end{proof}

\begin{lemma}\label{lem:Kato}
Let $\delta \in (0,\frac{d-\alpha}{2}]\cap (0,\alpha)$. There is a constant $c>0$ such that for all $t>0$, $x,y\in\Rdz$,
\begin{align*}
\int_0^t \int_{\Rd} \tilde{p}^{\,0}_{V^0_\delta}(s,x,z) |z|^{\delta-\alpha}\, \tilde{p}^{\,0}_{V^0_\delta}(t-s,z,y)\, \d{z} \d{s} \le ct^{\delta/\alpha} \tilde{p}^{\,0}_{V_\delta^0}(t,x,y)\,. 
\end{align*}
\end{lemma}
\begin{proof}
Recall from
\cite[Theorem~1.1 and (2.4)]{MR3933622}
that
\begin{align}\label{approx:BGJP}
\tilde{p}^{\,0}_{V^0_\delta}(t,x,y) \approx p^{0}(t,x,z)H^0(t,x)H^0(t,y)\,,
\end{align}
where $H^0(t,x)=1+t^{\delta/\alpha}|x|^{-\delta}$
as introduced in Section~\ref{sec:main}.
Note that for $0<s<t$ we have $H^0(t-s,y)\leq H^0(t,y)$ and
\begin{align*}
H^0(s,x)H^0(s,z)H^0(t-s,z)
&\leq H^0(s,z) H^0(s,x)H^0(t,z)\\
&\leq H^0(s,z) \left(H^0(t,x) H^0(s,z)+\frac{t^{\delta/\alpha}}{|z|^\delta} \right)\\
&=
H^0(t,x)\left(1+\frac{2s^{\delta/\alpha}}{|z|^\delta} +\frac{s^{2\delta/\alpha}}{|z|^{2\delta}}\right)
+ t^{\delta/\alpha} \left( \frac{1}{|z|^\delta} + \frac{s^{\delta/\alpha}}{|z|^{2\delta}}\right).
\end{align*}
By Lemma~\ref{lem:int-aux} we get for $k\in\{0,1,2\}$,
$$
\int_0^t \int_{\Rd} s^{k\delta/\alpha} p^{0}(s,x,z)|z|^{-k\delta}|z|^{\delta-\alpha} \,\d{z}\d{s}
\leq c t^{\delta/\alpha}\,,
$$
and for $k\in\{0,1\}$,
$$
\int_0^t \int_{\Rd} s^{k \delta/\alpha} p^{0}(s,x,z)|z|^{-(k+1)\delta} |z|^{\delta-\alpha}\, \d{z}\d{s}
\leq c  \log\left(1+\left(t|x|^{-\alpha}\right)^{k\delta/\alpha+1}\right)\leq c H^0(t,x)\,.
$$
Thus
\begin{align*}
\int_0^t \int_{\Rd}p^{0}(s,x,z)H^0(s,x)H^0(s,z)H^0(t-s,y)H^0(t-s,z) |z|^{\delta-\alpha} \,\d{z}\d{s}
\leq c t^{\delta/\alpha} H^0(t,x) H^0(t,y)\,.
\end{align*}
Similarly,
\begin{align*}
\int_0^t \int_{\Rd}p^{0}(t-s,z,y)H^0(s,x)H^0(s,z)H^0(t-s,y)H^0(t-s,z) |z|^{\delta-\alpha} \d{z}\d{s} \leq c t^{\delta/\alpha}H^0(t,x)H^0(t,y)\,.
\end{align*}
Finally, the result follows from the latter two inequalities combined with
\begin{align*}
&\tilde{p}^{\,0}_{V^0_\delta}(s,x,z)
\tilde{p}^{\,0}_{V^0_\delta}(t-s,z,y)\\
&\leq c p^{0}(t,x,y)\Big(p^{0}(s,x,z)+
p^{0}(t-s,z,y)\Big)
H^0(s,x)H^0(s,z)H^0(t-s,y)H^0(t-s,z)\,,
\end{align*}
which holds by \eqref{approx:BGJP} and 3G-inequality for 
$p^0(t,x,y)$, see 
\cite[Theorem~4]{MR2283957}.
\end{proof}

We show that $V_\delta - V^0_\delta$
can be conveniently used to perturb $\tilde{p}^{\, 0}_{V_\delta^0}$.
\begin{corollary}\label{cor:Kato}
Let $\delta \in (0,\frac{d-\alpha}{2}]$.
For every $T>0$ there is a constant $c>0$ such that for all $t\in (0,T]$ and $x,y\in\Rdz$,
\begin{align*}
\int_0^t \int_{\Rd} \tilde{p}^{\,0}_{V^0_\delta}(s,x,z) \Big(V_\delta(z) - V^0_\delta(z)\Big) \tilde{p}^{\,0}_{V^0_\delta}(t-s,z,y)\, \d{z} \d{s} \leq c
\max\{t, t^{\delta/\alpha}\}\, \tilde{p}^{\,0}_{V_\delta^0}(t,x,y)\,. 
\end{align*}
\end{corollary}
\begin{proof}
Let $q(z):=V_{\delta}(z)-V^0_\delta(z)=V_{\delta}(z)-\kappa_\delta |z|^{-\alpha} \geq 0$,
see Corollary~\ref{cor:diff_nonneg}.
If $\delta \geq \alpha$ the inequality is trivial, because due to 
Lemmas~\ref{lem:V-decreasing} and~\ref{lem:diff_est} the function
$q$ is bounded, and $\tilde{p}^{\, 0}_{V_\delta^0}$ satisfies Chapman-Kolmogorov equation \eqref{eq:C-K}.
If $0<\delta<\alpha$, then $q$ is bounded for $|z|\geq 1/2$, while for $|z|\leq 1/2$ we can use
Lemmas~\ref{lem:diff_est}
and~\ref{lem:Kato} to obtain the desired inequality.
\end{proof}

\begin{proposition}\label{prop:upper_bound}
Let $\delta \in (0,d-\alpha)$.
For every $T>0$ there is a constant $c>0$ such that for all $t\in (0,T]$ and $x,y\in\Rdz$,
$$
\tilde{p}^{\,0}_{V_\delta^0}(t,x,y) \leq \tilde{p}^{\,0}_{V_{\delta}}(t,x,y)\leq 
c \,\tilde{p}^{\,0}_{V_\delta^0}(t,x,y)\,.
$$
\end{proposition}
\begin{proof}
According to
\eqref{eq:pert-of-pert}
we treat $\tilde{p}^{\,0}_{V_{\delta}}$
as the  perturbation of
$\tilde{p}^{\,0}_{V^0_\delta}$ by 
$q_\delta(z):=V_{\delta}(z)-V^0_\delta(z)=V_{\delta}(z)-\kappa_\delta |z|^{-\alpha} \geq 0$,
see Corollary~\ref{cor:diff_nonneg}.
If $\delta\in (0,\frac{d-\alpha}{2}]$
the result follows from
\cite[Theorem~2]{MR2457489} since $q$ is relatively Kato for $\tilde{p}^{\,0}_{V_\delta^0}$ by
Corollary~\ref{cor:Kato}.
If $\delta\in [\frac{d-\alpha}{2},d-\alpha)$ 
the same argument applies, because
$0\leq q_\delta\leq q_{\delta'}$
with $\delta'=d-\alpha-\delta\in (0,\frac{d-\alpha}{2}]$, see
Remark~\ref{rem:reflect_beta}.
The lower bound follows from the trivial inequality $\tilde{p}^{\,0}_{V^0_\delta}\leq \tilde{p}^{\,0}_{V_\delta}$, see the comment preceding Lemma~\ref{lem:int-aux}.
\end{proof}

We will use the following observation several times throughout the paper.
\begin{remark}\label{rem:often}
Let $\delta \in (0,d-\alpha)$.
Recall that $\tilde{p}:= \tilde{p}_{V_{\delta}}$
 is the Schr{\"o}dinger perturbation of $p:=p^1$ by $V_\delta:=V^1_\delta$.
Directly from \eqref{def:p^m} we have $p\leq p^0$, and by Proposition~\ref{prop:upper_bound},
$$
\tilde{p}\leq 
\tilde{p}^{\,0}_{V_{\delta}}\leq 
c\, \tilde{p}^{\,0}_{V_\delta^{0}}
$$
holds on $(0,T]\times\Rdz\times\Rdz$,
whenever $\delta \in (0,d-\alpha)$ and $T>0$ are fixed.
\end{remark}

We are ready to show Proposition~\ref{prop:weak_sol}.
The essence 
of the
proof is to take a similar equality 
for the original transition density $p$
and the operator 
$\partial_u - (-\Delta_z+1)^{\alpha/2}$,
and to use
the algebraic structure of the
perturbed transition density $\tilde{p}$.
That idea goes back to \cite[(39)]{MR2457489},
\cite[Lemma~4]{MR3000465}
and \cite[Lemma~2.1]{MR3514392}.
In doing so one should ensure that certain integrals converge absolutely. This can be guaranteed using Proposition~\ref{prop:upper_bound}
and
\cite[Proposition~3.2]{MR3933622}.

\begin{proof}[Proof of Proposition~\ref{prop:weak_sol}]
Recall that $\tilde{p}:= \tilde{p}_{V_{\delta}}$
that is the perturbation of $p:=p^1$ by $V_\delta:=V^1_\delta$.
For $s\in\R$ and $x\in\Rdz$ we define
\begin{align*}
Pf(s,x)&=\int_s^{\infty}\int_{\Rd} p(u-s,x,z)f(u,z)\,\d{z}\d{u}\,,\\
\tilde{P}f(s,x)&=\int_s^{\infty}\int_{\Rd} \tilde{p}(u-s,x,z)f(u,z)\,\d{z}\d{u}\,,\\
V_\delta f(s,x)&=V_\delta(x)f(s,x)\,.
\end{align*}
for a jointly measurable function $f$ such that the integrals converge absolutely.
Let $\psi(u,z)=\left(\partial_u - (-\Delta_z+1)^{\alpha/2}\right)\phi(u,z)$.
It is well known that 
for all $s\in\R$, $x\in\Rd$ and
$\phi \in C_c^{\infty}(\R\times\Rd)$ we have (see \cite[Theorem~4.1 and (1.16)]{MR3514392})
\begin{align*}
\int_s^{\infty}\int_{\Rd} p(u-s,x,z)
\left[\partial_u - (-\Delta_z+1)^{\alpha/2} \right] \phi(u,z) \,\d{z}\d{u}
 =-\phi(s,x)\,.
\end{align*}
Hence,
$P\psi=-\phi$.
Then,
$$
\tilde{P}(\psi+V_\delta \phi)=
(P+\tilde{P}V_\delta P)\psi
+\tilde{P}V_\delta \phi
=-\phi+\tilde{P}V_\delta P \psi
+\tilde{P}V_\delta (-P\psi)=-\phi\,,
$$
which is the desired equality, but we  need to make sure that the integrals converge absolutely.
We consider
$$
\tilde{P}^0 f(s,x)=\int_s^{\infty}\int_{\Rd} \tilde{p}^{\,0}_{V_\delta^0}(u-s,x,z)f(u,z)\,\d{z}\d{s}\,.
$$
Note that the function $\psi$ is bounded 
\cite[p. 211]{MR1739520}
and zero if $|u|$ is large.
By Remark~\ref{rem:often}, for every $s\in\R$ and $x\in\Rdz$ there is $c$ such that
$$
\tilde{P}|\psi| (s,x)\leq c \tilde{P}^0|\psi|(s,x)<\infty\,.
$$
The finiteness follows from
\cite[Proposition~3.2]{MR3933622}
and because
$\tilde{P}^0$ is the same for $\delta$ and $\delta'=d-\alpha-\delta$,
see Remark~\ref{rem:reflect_beta}.
This implies that $P|\psi|$, $\tilde{P}|\psi|$, $\tilde{P}V_\delta P|\psi|$ and $\tilde{P}V_\delta|\phi|$ are finite.
Indeed, by \eqref{eq:Duh} we have
$$
P|\psi|+ \tilde{P}V_\delta P|\psi|
=\tilde{P}|\psi|<\infty
\qquad \mbox{and}\qquad
\tilde{P}V_\delta|\phi|\leq \tilde{P}V_\delta P|\psi|<\infty\,. 
$$
\end{proof}

\begin{proposition}\label{prop:m-to-zero}
Let $\delta\in(0,d-\alpha)$.
If $m\downarrow 0$, then for all $t>0$, $x,y\in\Rdz$,
$$
p^m(t,x,y) \uparrow p^0(t,x,y)\,,\qquad\quad
V_\delta^m (x) \downarrow V_\delta^0(x)\,,\qquad\quad
\tilde{p}^{\,m}(t,x,y)\to \tilde{p}^{\,0}(t,x,y)\,.
$$
\end{proposition}
\begin{proof}
The convergence and monotonicity of $p^m$ follows from \eqref{def:p^m}. 
The convergence of $V_\delta^m$ stems from Lemmas~\ref{lem:lim_V} and~\ref{lem:scal-m}, and
the monotonicity from \eqref{properties}.
To prove the third convergence we use the dominated convergence theorem. To justify its use we note that
 $p^m\leq p^0$ and $V_\delta^m\leq V_\delta^1$
if $m\in (0,1]$,
and we let
$\tilde{p}^{\,0}_{V_\delta}$
(the Schr{\"o}dinger
perturbation of $p^{0}$ by $V_\delta=V_\delta^1$)
to be the majorant, see also
Proposition~\ref{prop:upper_bound} for summability or integrability.
\end{proof}

\section{Heat kernel of $-(-\Delta+1)^{\alpha/2}+V_{\delta}$}
\label{sec:our}

Under the constraint \fbox{in the whole section that $\alpha\in(0,2\land d)$ and $\delta\in (0,\tfrac{d-\alpha}{2}]$} 
we finally consider
the heat kernel $\tilde{p}$ for
the operator \eqref{def:operator}. As specified after Theorem~\ref{thm:main-1} and in the comment preceding Lemma~\ref{lem:scal-m-tp}, we investigate
$\tilde{p}:= \tilde{p}_{V_{\delta}}$
that is the Schr{\"o}dinger perturbation of $p:=p^1$ by $V_\delta:=V^1_\delta$.

\subsection{Upper bound}\label{subsec:upper_bound}

We consider the following function
$$
H(t,x):=1+\frac{h_\delta(x)}{h_\delta(t^{1/\alpha})}\,.
$$
Note that for every $T>0$ we have 
on $(0,T]\times\Rdz$
(see Lemma~\ref{lem:h-expl}  and \eqref{eq:lim_K}),
\begin{align}\label{approx:H_H0}
H(t,x)\approx H^0(t,x)\,.
\end{align}

\begin{proposition}\label{prop:|x-y|-small}
For every $T,R>0$ there exists a constant $c$ such that for all
$t\in (0,T]$ and $x,y\in\Rdz$ satisfying $|x-y|\leq R$,
$$
\tilde{p}(t,x,y)\leq c p(t,x,y)H(t,x)H(t,y)\,.
$$
\end{proposition}
\begin{proof}
By
Remark~\ref{rem:often}
and \eqref{approx:BGJP} we have
$
\tilde{p}(t,x,y) \leq c
p^0(t,x,y)H^0(t,x)H^0(t,y)
$.
Now, due to \eqref{approx:H_H0} it suffices to observe that
$
p^0(t,x,y)\leq
c( t^{-d/\alpha}\land \frac{t}{|x-y|^{d+\alpha}})
\leq c p(t,x,y)\,,
$
whenever $t\in (0,T]$ and $|x-y|\leq R$, see  
e.g. \cite[(2.4)]{MR3933622} and
\eqref{approx:p-stable}.
\end{proof}

Recall that $V_\delta$ is 
radial decreasing (see Lemma~\ref{lem:V-decreasing}) thus we can define the inverse of its radial profile, which we denote by $V_\delta^{-1}$.

\begin{lemma}
For all $t>0$ and $x,y\in\Rdz$,
\begin{align}\label{ineq:tp-cut_ball}
\tilde{p}(t,x,y)\leq 2p(t,x,y)+2\int_0^t\int_{|z|\leq V_\delta^{-1}(\frac1{2t})}\tilde{p}(t-s,x,z) V_\delta(z) p(s,z,y)\,\d{z}\d{s}\,.
\end{align}
\end{lemma}
\begin{proof}
By Duhamel's formula \eqref{eq:Duh},
the monotonicity of $V_\delta$
and Chapman-Kolmogorov equation \eqref{eq:C-K},
for any $R> 0$ we have
\begin{align*}
\tilde{p}(t,x,y) &= p(t,x,y)+\int_0^t\int_{|z|\leq R} \tilde{p}(t-s,x,z)V_\delta(z)p(s,z,y)\,\d{z}\d{s}\\
&\qquad\qquad\qquad+\int_0^t\int_{|z|> R} \tilde{p}(t-s,x,z)V_\delta(z)p(s,z,y)\,\d{z}\d{s}\\
&\leq
p(t,x,y)+\int_0^t\int_{|z|\leq R} \tilde{p}(t-s,x,z)V_\delta(z)p(s,z,y)\,\d{z}\d{s}
+ t V_\delta(R) \tilde{p}(t,x,y)\,.
\end{align*}
It suffices to take $V_\delta(R)= 1/(2t)$.
\end{proof}

\begin{proposition}\label{prop:tp-mass_upper}
Let $T>0$. There exists a constant $c$ such that for all $t\in(0,T]$, $x\in\Rdz$,
\begin{align}\label{ineq:tp-mass_upper}
\int_{\Rd} \tilde{p}(t,x,y)\,\d{y}\leq c H(t,x)\,.
\end{align}
\end{proposition}
\begin{proof}
Integrating \eqref{ineq:tp-cut_ball} we get
\begin{align*}
\int_{\Rd} \tilde{p}(t,x,y)\,\d{y}\leq 2 +
2\int_0^t\int_{|z|\leq V_\delta^{-1}(\frac1{2t})}\tilde{p}(t-s,x,z) V_\delta(z)\,\d{z}\d{s}\,.
\end{align*}
Note that by Lemma~\ref{lem:lim_V}
there is $c_1$ such that
$V_\delta^{-1}(\frac1{2t})\leq c_1 t^{1/\alpha}$
for all $t\in (0,T]$.
Thus, by \eqref{approx:p-stable}
there is $c_2>0$ such that
for all $0<s<t\leq T$ and
$|z|\leq V_\delta^{-1}(\frac1{2t})$,
$$
c_2\leq \int_{|y|\leq 2c_1t^{1/\alpha}} p(s,z,y)\,\d{y}\,.
$$
Then, by \eqref{eq:Duh} in the second inequality,
\begin{align*}
\int_{\Rd} \tilde{p}(t,x,y)\,\d{y}
&\leq 2 + \frac{2}{c_2}\int_{|y|\leq 2c_1t^{1/\alpha}} \left( \int_0^t\int_{|z|\leq V_\delta^{-1}(\frac1{2t})}\tilde{p}(t-s,x,z) V_\delta(z) p(s,z,y) \,\d{z}\d{s}\right)\d{y}\\
&\leq 2 + \frac{2}{c_2}\int_{|y|\leq 2c_1 t^{1/\alpha}} \tilde{p}(t,x,y)\,\d{y}\,.
\end{align*}
Now, by the monotonicity of $h_\delta$ and Corollary~\ref{cor:BDK} we get
$$
\int_{|y|\leq 2c_1 t^{1/\alpha}} \tilde{p}(t,x,y)\,\d{y}
\leq 
\frac1{h_\delta(2c_1t^{1/\alpha})}\int_{|y|\leq 2c_1 t^{1/\alpha}} \tilde{p}(t,x,y)h_\delta(y)\,\d{y}\leq \frac{h_\delta(x)}{h_\delta(2c_1t^{1/\alpha})}\,.
$$
Finally, note that $h_\delta(2c_1t^{1/\alpha})\approx h_\delta(t^{1/\alpha})$ for $t\in(0,T]$.
\end{proof}

In what follows, for $T>0$ we set
$$R_0:=\max\{1, V_\delta^{-1}(1/(2T))\}\,.$$
We shall also use $\nu(x)$
defined in Corollary~\ref{thm:3}.

\begin{lemma}\label{lem:|x|-small_|y|-large}
Let $T>0$. There exists a constant $c$ such that for all 
$t\in(0,T]$,
$R\in [R_0,2R_0]$ and
$0<|x|\leq R$, $|y|\geq R+1$,
$$
\tilde{p}(t,x,y)\leq c t\nu(x-y)H(t,x)\,.
$$
\end{lemma}
\begin{proof}
Since $|x|\leq 2R_0$, $|y|\geq R_0+1$,
using \eqref{approx:p}
and \eqref{ineq:nu-shift-2}
we get
$p(s,z,y)\leq c s \nu(x-y)$
for all $s\in (0,T]$ and $|z|\leq R_0$.
Thus, by
\eqref{ineq:tp-cut_ball} we have
\begin{align*}
\tilde{p}(t,x,y)
&\leq 2p(t,x,y)+2\int_0^t\int_{|z|\leq R_0}\tilde{p}(t-s,x,z) V_\delta(z) p(s,z,y)\,\d{z}\d{s}\\
&\leq ct\nu(x-y)+ct\nu(x-y)
\int_0^t\int_{|z|\leq R_0}\tilde{p}(t-s,x,z) V_\delta(z)\,\d{z}\d{s}\,.
\end{align*}
Furthermore, by \eqref{eq:Duh} and \eqref{ineq:tp-mass_upper} we have
$$
\int_0^t\int_{\Rd}\tilde{p}(t-s,x,z) V_\delta(z)\,\d{z}\d{s}
\leq \int_{\Rd} \tilde{p}(t,x,y)\,\d{y}
\leq c H(t,x) \,.
$$
\end{proof}

\begin{proposition}\label{prop:|x-y|-large}
Let $T>0$. There exists a constant $c$ such that for all $t\in (0,T]$ and $|x|\vee |y| \geq 3R_0$,
$$
\tilde{p}(t,x,y)\leq c t\nu(x-y)H(t,x)H(t,y)\,.
$$
\end{proposition}
\begin{proof}
By the symmetry of $\tilde{p}$ we may and do assume that $|y|\geq 3R_0$. The case of $|x|\leq 2R_0$ is covered by Lemma~\ref{lem:|x|-small_|y|-large} with $R=2R_0$. 
From now on we consider
$|x|\geq 2R_0$.
Using the symmetry of $\tilde{p}$ again, 
we get by \eqref{ineq:tp-cut_ball} that
\begin{align*}
\tilde{p}(t,x,y)
&\leq 2p(t,y,x)+2\int_0^t\int_{|z|\leq R_0}\tilde{p}(t-s,z,y)  V_\delta(z) p(s,z,x) \,\d{z}\d{s}\,.
\end{align*}
Furthermore, by Lemma~\ref{lem:|x|-small_|y|-large} with $R=R_0$, we have 
by the monotonicity of $h_\delta$
that for $0<s<t\leq T$ and $|z|\leq R_0$,
$$
\tilde{p}(t-s,z,y)\leq c T \nu(z-y) H(T,z)\,.
$$
Thus, by \eqref{approx:p} and \eqref{ineq:nu-prod} we have
\begin{align*}
\tilde{p}(t,x,y)
&\leq c t\nu(x-y)+ c T^2 \int_0^t\int_{|z|\leq R_0}  \nu(z-y)  H(T,z)  V_\delta(z) \nu(x-z) \,\d{z}\d{s}\\
&\leq c t\nu(x-y)+c T^2\nu(x-y)
\int_0^t\int_{|z|\leq R_0}  H(T,z)  V_\delta(z) \,\d{z}\d{s}\\
&=
c t\nu(x-y)+ct\nu(x-y)\,
T^2 \int_{|z|\leq R_0}  H(T,z)  V_\delta(z) \,\d{z}\,.
\end{align*}
Finally, notice that
$\int_{|z|\leq R_0}  H(T,z)  V_\delta(z) \,\d{z}
<\infty$ since $\alpha<d$,
and $H(t,x)H(t,y)\geq 1$.
\end{proof}

\begin{theorem}\label{thm:tp-upper}
Let $T>0$. There exists a constant $c$ such that for all $t\in(0,T]$ and $x,y\in\Rdz$,
$$
\tilde{p}(t,x,y)\leq c p(t,x,y)H(t,x)H(t,y)\,.
$$
\end{theorem}
\begin{proof}
If $|x-y|\leq 6R_0$, the inequality follows from Proposition~\ref{prop:|x-y|-small}. If $|x-y|\geq 6R_0$, then $|x|\vee |y|\geq 3R_0$ and the inequality holds by Proposition~\ref{prop:|x-y|-large} and \eqref{approx:p-tnu}.
\end{proof}

\subsection{Invariance of $h_\delta$}

We show that in fact there is  equality in 
Corollary~\ref{cor:BDK} if $\delta\in (0,\frac{d-\alpha}{2}]$.
Recall that the functions $h_\beta:=h^1_\beta$ and $\bar{h}_\beta:=\bar{h}^1_\beta$ are defined in
Subsection~\ref{subsec:derivation}.

\begin{lemma}\label{lem:N-L}
For all $\beta\in(0,d-\alpha)$ and $t>0$, $x\in \Rdz$,
$$
\int_{\R^d} p(t,x,y) h_{\beta}(y)\,\d{y}
+ \int_0^t \int_{\R^d} p(s,x,y)  \bar{h}_{\beta}(y)\, \d{y}\d{s}
 =h_{\beta}(x)\,.
$$
\end{lemma}
\begin{proof}
Let $f(t)=t^{\frac{d-\alpha-\beta}{\alpha}}$. 
Then,
\begin{align*}
\int_0^t \int_{\R^d} p(s,x,y)\bar{h}_{\beta}(y)\,\d{y}\d{s}
&= \int_0^t \int_0^{\infty} p(s+r,x,0) f'(r)\,\d{r}\d{s}\\
&=-\int_0^t \int_0^{\infty}\partial_s\, p(s+r,x,0) f(r)\,\d{r}\d{s}\\
&=\int_0^{\infty} [p(r,x,0)-p(t+r,x,0)]f(r)\,\d{r}\\
& = h_{\beta}(x)-\int_{\Rd}p(t,x,y) h_{\beta}(y)\,\d{y}\,.
\end{align*}
The first equality above follows from the Chapman-Kolmogorov equation and the definition of $\bar{h}_{\beta}$.
In the second equality we used integration by parts
and 
that 
$p(s+r,x,0)f(r)\leq p^0(s+r,x,0)f(r)$ vanishes at zero and at infinity as a function of $r$.
In the third equality we used
Fubini's theorem, which was justified because
$\int_0^t \int_0^{\infty} |\frac{\partial}{\partial s}\, p(s+r,x,0) f(r)|\,\d{r}\d{s}<\infty$,
see Lemma~\ref{lam:p-deriv-bound}.
\end{proof}

Here is how Lemma~\ref{lem:N-L} propagates onto $\tilde{p}$.
\begin{proposition}\label{prop:N-L}
For all $\beta\in(0,d-\alpha)$ and $t>0$, $x\in \Rdz$,
\begin{align*}
\int_{\Rd} \tilde{p}(t,x,y) h_{\beta}(y)\,\d{y}
+
\int_0^t \int_{\Rd}  \tilde{p}(s,x,y)
& \bar{h}_{\beta}(y)\, \d{y}\d{s}\\
&=h_{\beta}(x)
+\int_0^t \int_{\Rd} \tilde{p}(s,x,y) V_\delta(y)h_{\beta}(y) \,\d{y}\d{s}\,.
\end{align*}
\end{proposition}
\begin{proof}
By Duhamel's formula \eqref{eq:Duh},
\begin{align*}
&\int_{\Rd} \tilde{p}(t,x,y) h_{\beta}(y)\,\d{y}\\
&\qquad =\int_{\Rd} p(t,x,y) h_{\beta}(y)\,\d{y}
 +
\int_0^t\int_{\Rd} \tilde{p}(s,x,z)V_\delta(z)\int_{\Rd} p(t-s,z,y)h_\beta(y)\,\d{y}\,\d{z}\d{s}\,,
\end{align*}
and
\begin{align*}
&\int_0^t \int_{\Rd} \tilde{p}(s,x,y)
\bar{h}_{\beta}(y)\, \d{y}\d{s}\\
&\qquad=
\int_0^t \int_{\Rd} p(s,x,y) \bar{h}_{\beta}(y)\,\d{y}\d{s}
+
\int_0^t \int_{\Rd}
\int_0^s \int_{\Rd} \tilde{p}(u,x,z)V_\delta(z)p(s-u,z,y)\,\d{z}\d{u}\, \bar{h}_{\beta}(y) \,\d{y}\d{s}\\
&\qquad=\int_0^t \int_{\Rd} p(s,x,y) \bar{h}_{\beta}(y)\,\d{y}\d{s}
 +
\int_0^t \int_{\Rd}
 \tilde{p}(u,x,z)V_\delta(z)
\int_0^{t-u} \int_{\Rd} 
 p(r,z,y) \bar{h}_{\beta}(y) \, \d{y}\d{r} \, \d{z}\d{u}  \,.
\end{align*}
It suffices to add
the above expressions and apply Lemma~\ref{lem:N-L}.
\end{proof}

Recall that we assume $\delta\in (0,\frac{d-\alpha}{2}]$.
Given the equality in Proposition~\ref{prop:N-L}, we would like to put $\beta=\delta$ and to cancel the double integrals, because $\bar{h}_\delta(x)=V_\delta(x) h_\delta(x)$, $x\in\Rdz$. That would require the finiteness of these integrals. We also note that the case $\delta=\frac{d-\alpha}{2}$ is more challenging than that of $\delta<\frac{d-\alpha}{2}$. We prove both cases at once by a limiting procedure.

\begin{theorem}\label{thm:tph=h}
For all $t>0$ and $x\in \Rdz$,
\begin{align*}
\int_{\R^d} \tilde{p}(t,x,y) h_{\delta}(y)\,\d{y}
=h_{\delta}(x)\,.
\end{align*}
\end{theorem}
\begin{proof}
By the definition of $h_\beta^m$ in Subsection~\ref{subsec:derivation} we have
\begin{align*}
h_\beta(y)\leq h_\beta^0(y)= \frac{2^{\beta}\Gamma\left(\frac{d-\beta}{\alpha}\right)\Gamma\left(\frac{\beta}{2}\right) }{(4\pi)^{d/2} \Gamma\left(\frac{d-\beta}{2}\right)} |y|^{-\beta}\,.
\end{align*}
Thus, by Remark~\ref{rem:often} and \cite[(3.3)]{MR3933622},
for $0<\beta<\delta$
the integrals
on the left hand side of the equality
in Proposition~\ref{prop:N-L} are finite,
hence the one on the right hand side is also finite, and we can rewrite the equality as
\begin{align*}
\int_{\R^d} \tilde{p}(t,x,y) h_{\beta}(y)\,\d{y}
=h_{\beta}(x)
+\int_0^t \int_{\R^d} \tilde{p}(s,x,y)
\Big( V_\delta(y)- V_\beta(y) \Big) h_{\beta}(y) \,\d{y}\d{s}\,.
\end{align*}
Fix $0<\beta_0<\alpha\land \delta$. 
For $\beta\in [\beta_0,\delta]$
we have
$h_\beta(x)\leq  \sup_{\beta\in [\beta_0,\delta]} h_\beta^0(1)  (|x|^{-\delta}+1)$,
therefore by Remark~\ref{rem:often} and
\cite[(3.2) and Proposition~3.2]{MR3933622}
we can apply the dominated convergence theorem to the integral on the left hand side,
$$
\lim_{\beta \uparrow \delta}\int_{\R^d} \tilde{p}(t,x,y) h_{\beta}(y)\,\d{y}=\int_{\R^d} \tilde{p}(t,x,y) h_{\delta}(y)\,\d{y}\,.
$$
Now, we split the integral on the right hand side,
\begin{align*}
\int_0^t \int_{\R^d} \tilde{p}(s,x,y)
\Big( V_\delta(y)- V_\beta(y) \Big) h_{\beta}(y) \,\d{y}\d{s}
= I_1+I_2+I_3+I_4\,,
\end{align*}
where
\begin{align*}
I_1 &=\int_0^t \int_{|y|\geq r_0} \tilde{p}(s,x,y)
\Big( V_\delta(y)- V_\beta(y) \Big) h_{\beta}(y) \,\d{y}\d{s}\,,\\
I_2 &= (\kappa_\delta-\kappa_\beta) \int_0^t \int_{|y|< r_0} \tilde{p}(s,x,y)
|y|^{-\alpha} h_{\beta}(y) \,\d{y}\d{s}\,,\\
I_3&=\int_0^t \int_{|y|<r_0} \tilde{p}(s,x,y)
\Big( V_\delta(y)-\kappa_\delta|y|^{-\alpha}\Big) h_{\beta}(y) \,\d{y}\d{s}\,,  \\
I_4 &= -\int_0^t \int_{|y|<r_0} \tilde{p}(s,x,y)
\Big( V_\beta(y) - \kappa_\beta |y|^{-\alpha} \Big) h_{\beta}(y) \,\d{y}\d{s}\,,
\end{align*}
and the value of $r_0\in (0,1]$ will be specified later.
Clearly, $I_2\geq 0$ and $I_3\geq 0$.
All the following inequalities will be uniform
in $\beta \in [\beta_0,\delta]$.
By Lemma~\ref{lem:diff_unif_bound}
 for all $0<|y|\leq 1$,
\begin{align*}
| V_\beta(y) - \kappa_\beta |y|^{-\alpha}| h_{\beta}(y)
\leq 
c |y|^{\beta_0-\alpha} \sup_{\beta\in [\beta_0,\delta]} h_\beta(1) |y|^{-\delta}
= c' |y|^{\beta_0-\delta-\alpha}\,.
\end{align*}
Fix $\varepsilon>0$.
Since 
by \cite[(3.3)]{MR3933622} with $\beta=\delta-\beta_0$
the latter expression is integrable against 
$\tilde{p}^{\,0}_{V_\delta^{0}}(s,x,y)\,\d{y}\d{s}$ on $(0,t]\times\Rd$,
by Remark~\ref{rem:often} there is $r_0\in (0,1]$ such that
$-I_4\leq \varepsilon$.
Now, for that choice of $r_0\in (0,1]$ and all
$|y|\geq r_0$,
$$
| V_\delta(y)- V_\beta(y)| h_{\beta}(y)
\leq \sup_{\beta\in [\beta_0, \delta]} \Big( V_\delta(r_0)+ V_\beta(r_0)\Big) h_{\beta}(r_0)
<\infty\,.
$$
Combined with \eqref{ineq:tp-mass_upper} it justifies the usage of the dominated convergence theorem, and we get
$\lim_{\beta \uparrow \delta} I_1=0$.
Finally, we have
\begin{align*}
\int_{\R^d} \tilde{p}(t,x,y) h_{\delta}(y)\,\d{y}
\geq h_{\delta}(x) - \varepsilon\,.
\end{align*}
Since $\varepsilon$ was arbitrary, and in view of
Corollary~\ref{cor:BDK}, we obtain the desired equality.
\end{proof}

\subsection{Lower bound}\label{subsec:lower_bound}

We adapt methods developed in \cite[Section~4.2]{MR3933622}. Let 
$$
\rho(t,x,y):=\frac{\tilde{p}(t,x,y)}{H(t,x)H(t,y)}
\qquad \mbox{and}\qquad \mu_t(\d{y}):=[H(t,y)]^2\d{y}\,.
$$
Using $\int_{\Rd} \tilde{p}(t,x,y) \d{y}\geq  \int_{\Rd}p(t,x,y) \d{y} \geq e^{-t}$ and Theorem~\ref{thm:tph=h} we get for all $t>0$, $x\in\Rdz$,
\begin{align}
&\int_{\Rd} \rho(t,x,z)\,\mu_t(\d{z})
= \frac1{H(t,x)} \int_{\Rd} \tilde{p}(t,x,y) H(t,y)\d{y} \nonumber \\
&= \frac1{H(t,x)}\left( \int_{\Rd}\tilde{p}(t,x,y)\d{y} + \frac1{h_\delta(t^{1/\alpha})} \int_{\Rd}\tilde{p}(t,x,y) h_\delta (y) \d{y} \right) \nonumber \\
&\geq \frac1{H(t,x)}\left(e^{-t}+\frac{h_\delta (x)}{h_\delta(t^{1/\alpha})}\right) \geq e^{-t}\,. \label{ineq:rho-1}
\end{align}
By Chapman-Kolmogorov equation \eqref{eq:C-K} and \eqref{approx:H_H0}, for all $t\in(0,T]$ and $x,y\in\Rdz$,
\begin{align}\label{ineq:rho-C-K}
\rho(2t,x,y) &\geq c \int_{\Rd}\rho(t,x,z)\rho(t,z,y)\, \mu_t(\d{z})\,. 
\end{align}

\begin{lemma}\label{lem:rho-ring_mass}
Let $T>0$. There are $r\in(0,1)$, $R\geq 2$ such that for all $t\in(0,T]$ and $0<|x|\leq (4t)^{1/\alpha}$,
$$
\int_{r\leq \frac{|z|}{t^{1/\alpha}}\leq R} \rho(t,x,z)\,\mu_t(\d{z}) \geq \frac{e^{-T}}{2} \,.
$$
\end{lemma}
\begin{proof}
By Theorem~\ref{thm:tp-upper}
we have $\rho(t,x,z)\leq c p(t,x,z)\leq c p^0(t,x,z)$, and $H(t,z)\leq c H^0(t,z)$.
Using the scaling of $p^0$ we get from  
\cite[Section~4.2]{MR3933622} that the integral of $c^3 p^0$ over $\{|z|\leq rt^{1/\alpha}\}\cup \{|z|\geq Rt^{1/\alpha}\}$
does not exceed $e^{-T}/2$. Therefore,
the result follows from \eqref{ineq:rho-1}.
\end{proof}

\begin{theorem}\label{thm:tp-lower}
Let $T>0$. There exists a constant $c$ such that for all $t\in(0,T]$ and $x,y\in\Rdz$,
$$
\tilde{p}(t,x,y)\geq c p(t,x,y)H(t,x)H(t,y)\,.
$$
\end{theorem}
\begin{proof}
If $|x|,|y|\geq t^{1/\alpha}$ we have $H(t,x)H(t,y)\leq 4$, so 
\begin{align}\label{ineq:st0}
\rho(t,x,y)\geq \frac14 \tilde{p}(t,x,y)\geq \frac14 p(t,x,y)\,.
\end{align}
Now, let $r$ and $R$ be taken from Lemma~\ref{lem:rho-ring_mass}, which will be used twice below.
We first assume that
$0<|x|\leq (4t)^{1/\alpha}$ and
$|y|\geq r (2t)^{1/\alpha}$.
Then by \eqref{ineq:rho-C-K}, \eqref{approx:H_H0}, \eqref{ineq:p-shift} and \eqref{approx:p} we get for all $t\in (0,T]$,
\begin{align*}
\rho(2t,x,y) &\geq c \int_{r\leq \frac{|z|}{t^{1/\alpha}}\leq R}\rho(t,x,z)\rho(t,z,y)\, \mu_t(\d{z}) 
\geq c  \int_{r\leq \frac{|z|}{t^{1/\alpha}}\leq R}\rho(t,x,z) \frac{p(t,z,y)}{[H^0(1,r)]^2} \,\mu_t(\d{z}) \nonumber \\
&\geq \frac{c}{[H^0(1,r)]^2}\left( \int_{r\leq \frac{|z|}{t^{1/\alpha}}\leq R}\rho(t,x,z) \,\mu_t(\d{z})\right) p(t,x,y) \geq c p(t,x,y) \geq c p(2t,x,y)\,.
\end{align*}
In particular,
for all $0<|x|\leq (2t)^{1/\alpha}$, $|y|\geq rt^{1/\alpha}$ and $t\in (0,T]$,
\begin{align}\label{ineq:st1}
\rho(t,x,y) &\geq c p(t,x,y)\,.
\end{align}
Now, we assume that $0<|x|, |y| \leq (2t)^{1/\alpha}$.
Then by \eqref{ineq:rho-C-K}, the symmetry and \eqref{ineq:st1} applied to $\rho(t,y,z)$,
and again \eqref{ineq:p-shift} and \eqref{approx:p},
we get for all $t\in (0,T]$,
\begin{align*}
\rho(2t,x,y) &\geq c \int_{r\leq \frac{|z|}{t^{1/\alpha}}\leq R}\rho(t,x,z)\rho(t,z,y)\, \mu_t(\d{z}) 
\geq c \int_{r\leq \frac{|z|}{t^{1/\alpha}}\leq R}\rho(t,x,z)p(t,z,y)\, \mu_t(\d{z})\\
&\geq c\left( \int_{r\leq \frac{|z|}{t^{1/\alpha}}\leq R}\rho(t,x,z)\, \mu_t(\d{z})\right) p(t,x,y)
\geq c p(t,x,y) \geq c p(2t,x,y)\,.
\end{align*}
In particular, for all $0<|x|, |y|\leq t^{1/\alpha}$
and $t\in (0,T]$,
\begin{align}\label{ineq:st2}
\rho(t,x,y) \geq c p(t,x,y)\,.
\end{align}
Note that \eqref{ineq:st1} covers the case
$0<|x|\leq t^{1/\alpha}$, $|y|\geq t^{1/\alpha}$.
Thus, due to the symmetry the inequalities \eqref{ineq:st0}, \eqref{ineq:st1} and
\eqref{ineq:st2} cover all cases.
\end{proof}

\subsection{Theorem~\ref{thm:MAIN} for $\delta \in (0,(d-\alpha)/2]$}\label{subsec:thm1-proof}

\begin{theorem}\label{thm:main-1}
Let $\alpha\in (0,2\land d)$ and $\delta \in (0, \frac{d-\alpha}{2}]$. 
The function $\tilde{p}$ is 
jointly continuous on $(0,\infty)\times\Rdz \times \Rdz$
and for every $T>0$ we have
$$
\tilde{p}(t,x,y)\approx p(t,x,y) H^0(t,x)H^0(t,y)
$$
on $(0,T]\times \Rdz \times \Rdz$.
The comparability constant can be chosen to depend only on $d,\alpha, \delta, T$.
\end{theorem}
\begin{proof}
The estimates follow from Theorems~\ref{thm:tp-upper} and~\ref{thm:tp-lower}, see \eqref{approx:H_H0}.
The continuity of $\tilde{p}$ 
essentially follows from
that of $p(t,x,y)$,
the Duhamel's formula
\eqref{eq:Duh}
and integrability properties
similarly to
\cite[Lemma~4.10]{MR3933622}, which are assured
by the inequalities given in Remark~\ref{rem:often}.
We omit details.
\end{proof}

\section{Further analysis and consequences}\label{sec:further}

In the whole section we assume that $\alpha\in(0,2\land d)$.

\subsection{Four transition densities}\label{subsec:4td}

Note that we have two heat kernels $p^0$ and $p:=p^1$
as well as two potentials $V^0_\delta$
and $V_\delta:=V^1_\delta$. 
That amounts to four possible transition densities as Schr{\"o}dinger perturbations.
In \cite{MR3933622}
the authors studied
$\tilde{p}^{\,0}_{V^0_\delta}$.
We have already discussed
$\tilde{p}^{\,0}_{V^1_\delta}$ for $\delta\in(0,d-\alpha)$
in Section~\ref{sec:largest}
and $\tilde{p}^{1}_{V^1_\delta}$
for $\delta\in(0,\frac{d-\alpha}{2}]$
in Section~\ref{sec:our}.
The remaining one is
$\tilde{p}^{1}_{V^0_\delta}$,
that is the one corresponding to the operator
$
-(-\Delta+1)^{\alpha/2}+\kappa_\delta |x|^{-\alpha}
$.

\begin{table}[h]
\begin{center}
\caption{Heat kernels, potentials and Schr{\"o}dinger perturbations.}
\begin{tabular}[h]{ c|c|l|l| m{1cm} }
\cline{2-4}
&\backslashbox{\small{Kernel}}{\small{Potential}} & \makecell{$V^0_\delta$} & \makecell{$V^1_\delta$} & \\ 
\cline{2-4}
\noalign{\vskip\doublerulesep
         \vskip-\arrayrulewidth}
\cline{2-4} 
& & & \\[-1pt]
&  $p^0(t,x,y)$ & \textcolor{Bittersweet}{$\tilde{p}^{\,0}_{V^0_\delta}$} & 
\textcolor{red}{$\tilde{p}^{\,0}_{V^1_\delta}$}  \\  [10pt]
\cline{2-4}
& & & \\[-1pt]
& $p^1(t,x,y)$ & \textcolor{OliveGreen}{$\tilde{p}^{1}_{V^0_\delta}$} & \textcolor{blue}{$\tilde{p}^{1}_{V^1_\delta}$} \\  [10pt]
\cline{2-4}
\end{tabular}
\end{center}
\end{table}

Several
trivial inequalities between transition densities in the table follow from 
$$p^1\leq p^0 \qquad \mbox{and}\qquad 0\leq V_\delta^0\leq V_\delta^1\,.$$
Namely, the function increases if we move to the right or upwards.

\begin{theorem}\label{thm:all_comp}
Let $\delta \in (0,d-\alpha)$.
For every $T>0$, we have on $(0,T]\times \Rdz\times \Rdz$ that
$$
\textcolor{red}{\tilde{p}^{\,0}_{V^1_\delta}(t,x,y)}
\approx
\textcolor{Bittersweet}{\tilde{p}^{\,0}_{V^0_\delta}(t,x,y)}
\approx p^0(t,x,y) H^0(t,x)H^0(t,y) \,,
$$
and
$$
\textcolor{blue}{\tilde{p}^{1}_{V^1_\delta}(t,x,y)}
\approx \textcolor{OliveGreen}{\tilde{p}^{1}_{V^0_\delta}(t,x,y)}
\approx p^1(t,x,y)H^0(t,x)H^0(t,y)
\,.
$$
Each function is jointly continuous on $(0,\infty)\times\Rdz\times\Rdz$.
\end{theorem}

The main new ingredient in the proof of Theorem~\ref{thm:all_comp} is Corollary~\ref{n-cor:Kato}, see below. It is a counterpart of
Corollary~\ref{cor:Kato} and
its proof is based on similar ideas used in Subsection~\ref{subsec:upper_bound}.

\begin{lemma}\label{n-prop:|x-y|-small}
Let $\delta\in (0,\frac{d-\alpha}{2}]$. For every $T,R>0$ there exists a constant $c$ such that 
for all $t\in (0,T]$ and $x,y\in\Rdz$ satisfying $|x-y|\leq R$,
\begin{align*}
\int_0^t \int_{\Rd} \tilde{p}^{1}_{V_\delta^1}(s,x,z)\Big(V^1_\delta(z) - V^0_\delta(z) \Big)
\tilde{p}^{1}_{V_\delta^1}(t-s,z,y)\,\d{z}\d{s}
\leq c \max\{t,t^{\delta/\alpha}\}\,\, \tilde{p}^{1}_{V_\delta^1}(t,x,y)\,.
\end{align*}
\end{lemma}
\begin{proof}
Using Remark~\ref{rem:often} and applying
Corollary~\ref{cor:Kato}
we can bound the left hand side by 
$c \max\{t, t^{\delta /\alpha}\} \tilde{p}^{\,0}_{V_\delta^0}$.
Now, exactly like in the proof of
Proposition~\ref{prop:|x-y|-small}
we get that $\tilde{p}^{\,0}_{V_\delta^0} \leq c \tilde{p}^{1}_{V_\delta^1}$
if $t\in (0,T]$ and $x,y\in\Rdz$ satisfy $|x-y|\leq R$.
\end{proof}

\begin{remark}\label{n-rem:|z|<1/2}
Let $\delta \in (0,\frac{d-\alpha}{2}]$. There exists a constant $c$ such that 
for all $t>0$ and $x,y\in\Rdz$,
\begin{align*}
\int_0^t \int_{|z|\geq 1/2} \tilde{p}^{1}_{V_\delta^1}(s,x,z)\Big(V^1_\delta(z) - V^0_\delta(z) \Big)
\tilde{p}^{1}_{V_\delta^1}(t-s,z,y)\,\d{z}\d{s}
\leq c t \, \tilde{p}^{1}_{V_\delta^1}(t,x,y)\,.
\end{align*}
The inequality simply follows from boundedness of $V^1_\delta - V^0_\delta$ for $|z|\geq 1/2$ and the Chapman-Kolmogorov equation \eqref{eq:C-K} for $\tilde{p}^{1}_{V_\delta^1}$.
\end{remark}

\begin{lemma}\label{n-lem:|x|-small_|y|-large}
Let $\delta\in(0,\frac{d-\alpha}{2}]$ and $0<\delta<\alpha$.
For every $T>0$ there exists a constant $c$ such that for all 
$t\in(0,T]$,
$R\in [1,2]$ and
$0<|x|\leq R$, $|y|\geq R+1$,
\begin{align*}
\int_0^t \int_{|z|\leq 1/2} \tilde{p}^{1}_{V_\delta^1}(s,x,z)\Big(V^1_\delta(z) - V^0_\delta(z) \Big)
\tilde{p}^{1}_{V_\delta^1}(t-s,z,y)\,\d{z}\d{s}
\leq c t^{\delta/\alpha+1} \nu(x-y)H^0(t,x)\,.
\end{align*}
\end{lemma}
\begin{proof}
We apply
Theorem~\ref{thm:tp-upper}, \eqref{approx:H_H0}
and
Lemma~\ref{lem:diff_est}
to obtain the following upper bound of the left hand side
\begin{align*}
c \int_0^t \int_{|z|\leq 1/2} p^{1}(s,x,z)H^0(s,x)H^0(s,z)|z|^{\delta-\alpha}\,
p^{1}(t-s,z,y)H^0(t-s,z)H^0(t-s,y)\,\d{z}\d{s}\,.
\end{align*}
Using $p^1(s,x,z)\leq p^0(s,x,z)$
and $p^1(t-s,z,y)\leq c t \nu(x-y)$, see \eqref{approx:p} and \eqref{ineq:nu-shift-2}, we further bound it by
\begin{align*}
&c t \nu(x-y) \int_0^t \int_{|z|\leq 1/2} p^{0}(s,x,z)H^0(s,x)H^0(s,z)H^0(t-s,z)H^0(t-s,y)|z|^{\delta-\alpha}\,\d{z}\d{s}\\
&\leq c t^{\delta/\alpha+1} \nu(x-y)  H^0(t,x) H^0(t,y)\,,
\end{align*}
where the latter follows from the same inequality as
in the proof of Lemma~\ref{lem:Kato}.
We finally notice that $H^0(t,y)$ is bounded under our assumptions.
\end{proof}

\begin{proposition}\label{n-prop:|x-y|-large}
Let $\delta\in(0,\frac{d-\alpha}{2}]\cap (0,\alpha)$.
For every $T>0$ there exists a constant $c$ such that for all $t\in (0,T]$ and $x,y\in\Rdz$ satisfying $|x|\vee |y| \geq 3$,
\begin{align*}
\int_0^t \int_{|z|\leq 1/2} \tilde{p}^{1}_{V_\delta^1}(s,x,z)\Big(V^1_\delta(z) - V^0_\delta(z) \Big)&
\tilde{p}^{1}_{V_\delta^1}(t-s,z,y)\,\d{z}\d{s}\\
&\leq c \max\{t,t^{\delta/\alpha}\} \, t\nu(x-y)H^0(t,x)H^0(t,y)\,.
\end{align*}
\end{proposition}
\begin{proof}
By the symmetry of $\tilde{p}^1_{V_\delta^1}$ we may and do assume that $|y|\geq 3$. The case of $|x|\leq 2$ is covered by Lemma~\ref{n-lem:|x|-small_|y|-large} with $R=2$. From now on we consider $|x|\geq 2$.
By Theorem~\ref{thm:tp-upper}, \eqref{approx:H_H0} and \eqref{approx:p}
for 
$0<s<t\leq T$ and
$0<|z|\leq 1/2$ 
we have
\begin{align*}
\tilde{p}^{1}_{V_\delta^1}(s,x,z)&\leq c s \nu(x-z) H^0(s,z) \,, \\
\tilde{p}^{1}_{V_\delta^1}(t-s,z,y)&\leq  c (t-s) \nu(z-y) H^0(t-s,z)\,,
\end{align*}
which together with Lemma~\ref{lem:diff_est} give us the following bound on the left hand side
\begin{align*}
c t^2 \int_0^t \int_{|z|\leq 1/2} \nu(x-z) H^0(s,z) |z|^{\delta-\alpha}\, \nu(z-y) H^0(t-s,z)
\,\d{z}\d{s}\,.
\end{align*}
Using \eqref{ineq:nu-prod} it can be further bounded by
\begin{align*}
c t^2 \nu(x-y)\, T\!\! \int_{|z|\leq 1/2} [H^0(T,z)]^2 |z|^{\delta-\alpha}
\,\d{z}\,.
\end{align*}
Finally, note that
$\int_{|z|\leq 1/2}  [H^0(T,z)]^2 |z|^{d-\alpha} \,\d{z}
<\infty$ since $\alpha<d$,
and  $H^0(t,x)H^0(t,y)\geq 1$.
\end{proof}

\begin{corollary}\label{n-cor:Kato}
Let $\delta\in (0,\frac{d-\alpha}{2}]$.
For every $T>0$ there is a constant $c>0$ such that for all $t\in (0,T]$ and $x,y\in\Rdz$,
\begin{align*}
\int_0^t \int_{\Rd} \tilde{p}^{1}_{V_\delta^1}(s,x,z)\Big(V^1_\delta(z) - V^0_\delta(z) \Big)
\tilde{p}^{1}_{V_\delta^1}(t-s,z,y)\,\d{z}\d{s}
\leq c \max\{t,t^{\delta/\alpha}\}\,\, \tilde{p}^{1}_{V_\delta^1}(t,x,y)\,.
\end{align*}
\end{corollary}
\begin{proof}
If $\delta \geq \alpha$ the inequality is trivial, because due to 
Lemma~\ref{lem:diff_est} the function
$V_\delta^1-V_\delta^0$ is bounded, and $\tilde{p}^{\, 1}_{V_\delta^1}$ satisfies Chapman-Kolmogorov equation \eqref{eq:C-K}.
Suppose that $0<\delta<\alpha$.
If $|x-y|\leq 6$, the inequality follows from Lemma~\ref{n-prop:|x-y|-small}. If $|x-y|\geq 6$, then also $|x|\vee |y|\geq 3$ and the inequality holds by
Remark~\ref{n-rem:|z|<1/2}, Proposition~\ref{n-prop:|x-y|-large}, \eqref{approx:p-tnu}, Theorem~\ref{thm:tp-lower} and \eqref{approx:H_H0}.
\end{proof}

Now we use
$V^0_\delta-V^1_\delta$
to perturb $\tilde{p}^{1}_{V_\delta^1}$
in order to analyse 
$\tilde{p}^{1}_{V^0_\delta}$
and to complement the picture.

\begin{proof}[Proof of Theorem~\ref{thm:all_comp}]
The estimates for 
$\tilde{p}^{\,0}_{V^0_\delta}$
with $\delta\in (0,d-\alpha)$ are given in
\cite[Theorem~1.1]{MR3933622}, see Remark~\ref{rem:reflect_beta}.
For
$\tilde{p}^{\,0}_{V^1_\delta}$
with $\delta\in (0,d-\alpha)$
the estimates are given in
Proposition~\ref{prop:upper_bound}.
For $\tilde{p}^{1}_{V^1_\delta}$ 
with $\delta \in (0,\frac{d-\alpha}{2}]$
they are given in Theorem~\ref{thm:main-1}.
For $\tilde{p}^{1}_{V^0_\delta}$ 
with $\delta \in (0,\frac{d-\alpha}{2}]$
the upper estimates follow from
$\tilde{p}^{1}_{V^0_\delta} \leq \tilde{p}^{1}_{V^1_\delta}$
and the lower bound 
$\tilde{p}^{1}_{V^0_\delta} \geq c \tilde{p}^{1}_{V^1_\delta}$
is due to
Corollary~\ref{n-cor:Kato}
and Lemma~\ref{lem:Kato-signed} with $q_1=V^1_\delta$ and $q_2=V^0_\delta-V^1_\delta$.
If $\delta \in [\frac{d-\alpha}{2},d-\alpha)$
the estimates follow from the case $\delta \in (0,\frac{d-\alpha}{2}]$,
because by Remark~\ref{rem:reflect_beta} we have
$$
\tilde{p}^{1}_{V^0_{\delta'}}=\tilde{p}^{1}_{V^0_{\delta}}
\leq \tilde{p}^{1}_{V^1_\delta}
\leq
\tilde{p}^{1}_{V^1_{\delta'}}\,,
$$
where $\delta'=d-\alpha-\delta\in (0,\frac{d-\alpha}{2}]$.
The proof of the continuity 
goes by similar lines to that in the
proof of Theorem~\ref{thm:main-1}, see
Subsection~\ref{subsec:thm1-proof}.
\end{proof}

\subsection{Blowup}
In this subsection we 
explain why the value \fbox{$\delta^*:=\frac{d-\alpha}{2}$}
of the parameter $\delta$ in $V_\delta$ can be regarded as critical. Note that out of the four transition densities discussed in Subsection~\ref{subsec:4td}
the function 
$\tilde{p}^{1}_{V^0_\delta}$
is the smallest and
$\tilde{p}^{\,0}_{V^1_\delta}$ is the largest.
We start with a blowup result.

\begin{corollary}
Let $q=(1+\varepsilon)V^0_{\delta^*}$ for $\varepsilon>0$. Then $\tilde{p}^1_q(t,x,y)=\infty$ for all $t>0$, $x,y\in \Rdz$.
\end{corollary}
\begin{proof}
Clearly, $\tilde{p}^1_q \geq \tilde{p}^1_{V^0_{\delta^*}}$. By
Theorem~\ref{thm:all_comp},
\begin{align}\label{ineq:for-blowup}
\tilde{p}^1_{V^0_{\delta^*}}(t,x,y)\geq c p^1(t,x,y) t^{(d-\alpha)/\alpha} |x|^{-(d-\alpha)/2} |y|^{-(d-\alpha)/2}\,.
\end{align}
According to \eqref{eq:pert-of-pert}, we consider 
$\tilde{p}^1_q$ as a perturbation of 
$\tilde{p}^1_{V^0_{\delta^*}}$ by $\varepsilon V^0_{\delta^*}$. Then, by the Duhamel's fromula
\eqref{eq:Duh}
and  \eqref{ineq:for-blowup} we get
\begin{align*}
\tilde{p}^1_q(t,x,y)&=
\tilde{p}^1_{V^0_{\delta^*}}(t,x,y)+
\int_0^t\int_{\Rd}
\tilde{p}^1_q(s,x,z)  \varepsilon V^0_{\delta^*}(z) \tilde{p}^1_{V^0_{\delta^*}}(t-s,z,y)\,\d{z}\d{s}
\\
&\geq \varepsilon \kappa_{\delta^*}
\int_0^t\int_{\Rd}
\tilde{p}^1_{V^0_{\delta^*}} (s,x,z)   |z|^{-\alpha} \,\tilde{p}^1_{V^0_{\delta^*}}(t-s,z,y)\,\d{z}\d{s}\\
&\geq  c (|x||y|)^{-(d-\alpha)/2} \int_0^t [s(t-s)]^{(d-\alpha)/\alpha} \int_{\Rd} 
p^1 (s,x,z)   |z|^{-d} \,p^1(t-s,z,y)\,\d{z}\d{s}  =\infty.
\end{align*}
In the last equality, we used the fact that for each $0<s<t$ and $x,y\in\Rdz$ there is $c>0$ such that $p^1(s,x,z)p^1(t-s,z,y)\geq c$ for all $|z|\leq 1$, and $\int_{|z|\leq 1} |z|^{-d}dz=\infty$.
\end{proof}

Now we give a no-blowup result.
Recall that $\kappa_\delta$ attains its maximum at $\delta=\delta^*$.

\begin{corollary}
Let $\delta\in (0,d-\alpha)$ and $\delta\neq \delta^*$.  If $\varepsilon\in (0,\frac{\kappa_{\delta^*}-\kappa_\delta}{\kappa_\delta})$ and $q=(1+\varepsilon)V^1_\delta$, then $\tilde{p}^{\,0}_q (t,x,y)<
\infty$ for  
all $t>0$, $x,y\in \Rdz$.
\end{corollary}
\begin{proof}
We have $q= V^0_\delta+(1+\varepsilon)(V^1_\delta-V^0_\delta) +\varepsilon V^0_\delta = q_1+q_2+q_3$. First, we construct $\tilde{p}^{\,0}_{q_1}=\tilde{p}^{\,0}_{V^0_\delta}$.
Next, we perturb it by $q_2$.
Proceeding like in the proof of Proposition~\ref{prop:upper_bound}, we see that $q_2$
is relatively Kato for $\tilde{p}^{\,0}_{q_1}$, and
thus by \cite[Theorem~2]{MR2457489} we get $\tilde{p}^{\,0}_{q_1+q_2}\leq (1+\theta) \tilde{p}^{\,0}_{q_1}$ for arbitrarily small  fixed $\theta>0$ on $(0,t_0]\times\Rdz\times\Rdz$.
We assure that $\varepsilon (1+\theta)   \leq  \frac{\kappa_{\delta^*}-\kappa_\delta}{\kappa_\delta}$ and
we perturb $\tilde{p}^{\,0}_{q_1+q_2}$ 
by $q_3$. Note that
$$
(1+\theta) q_3(x) \leq \frac{\kappa_{\delta^*}-\kappa_\delta}{\kappa_\delta} V^0_\delta(x) = (\kappa_{\delta^*}-\kappa_\delta)|x|^{-\alpha}=:q_4(x)\,.
$$
Therefore, by \eqref{ineq:tran_den}, on $(0,t_0]\times\Rdz\times\Rdz$ we get
$$
\tilde{p}^{\,0}_q=\widetilde{(\tilde{p}^{\,0}_{q_1+q_2})}_{q_3} 
\leq (1+\theta) \widetilde{(\tilde{p}^0_{q_1})}_{(1+\theta)q_3}
\leq (1+\theta) 
\widetilde{(\tilde{p}^{\,0}_{q_1})}_{q_4}
= (1+\theta) \tilde{p}^{\,0}_{V^0_{\delta^*}}\,.
$$
Finally, the finiteness holds for all $t>0$ by the latter inequality and Chapman-Kolmogorov equation \eqref{eq:C-K}.
\end{proof}

\appendix

\section{}\label{sec:app}

Despite the obvious collision we shall use the well established notation 
for the Bessel function~$K_\nu(r)$ and the L{\'e}vy measure $\nu(r)$. It is always self explanatory
which object is in use and should cause no confusion.

\subsection{The Bessel function}
\label{sec:ap-Bessel}

The modified Bessel function of the second kind $K_{\nu}$ is given by
$$
K_{\nu}(r)=\frac{1}{2} \left(\frac{r}{2}\right)^{\nu} \int_0^{\infty} u^{-\nu-1} e^{-u-\frac{r^2}{4u}} \,du\,,\quad \qquad \nu\in \R ,\,\, r>0\,.
$$
It holds that $K_{\nu}=K_{-\nu}$ and $K_{1/2}(x)=K_{-1/2}(x)=\sqrt{\pi/2}\, x^{-1/2} e^{-x}$.
The asymptotics of  $K_{\nu}$ at the origin and at infinity are well known,
$$\lim_{r\to 0^+}\dfrac{K_{0}(r)}{\log (1/r)}= 1\,,$$
\begin{align}\label{eq:lim_K}
\lim_{r\to 0^+}\frac{K_{\nu}(r)}{r^{-\nu}}= 2^{\nu-1}\Gamma(\nu),\qquad \mbox{for } \nu>0\,, 
\end{align}
and
$$\lim_{r\to \infty}{K_{\nu}(r)}\sqrt{r}e^r= \sqrt{\pi/2}, \qquad \mbox{for } \nu \geq 0\,.$$
It is not hard to see that $K_{\nu}(r)$ is decreasing in $r>0$.
We will use two representations
of the derivative of $K_\nu(r)$ for $r>0$,
see \cite[page~79]{MR0010746}
or \cite[(4.24)]{MR507514},
\begin{align}\label{def_K-1}
K_{\nu}'(r)=-K_{\nu+1}(r)+ (\nu/r) K_{\nu}(r)\,,
\end{align}
\begin{align}\label{def_K-2}
K_{\nu}'(r)=-K_{\nu-1}(r)-(\nu/r) K_{\nu}(r)\,.
\end{align}
From \cite[Lemma~2.2]{MR486686} we have that for $r>0$,
\begin{align}\label{K/K-1}
\nu \mapsto K_{\nu+1}(r)/K_{\nu}(r)\quad \mbox{is increasing}\,,
\end{align}
\begin{align}\label{K/K-2}
\nu \mapsto K_{\nu-1}(r)/K_{\nu}(r)\quad \mbox{is decreasing}\,.
\end{align}

\subsection{The relativistic stable kernel}\label{sec:rel-den}
We assume that $\alpha\in(0,2)$.
Recall that $\nu(x)=\nu^1(x)$ and $p(t,x,y)=p^1(t,x,y)$.
According to \cite[Example~2.3(a)]{TJ-KK-KS-2022} the function $\nu$ is an example of a density of a L{\'e}vy measure satisfying the assumptions \cite[(A1)--(A2)]{TJ-KK-KS-2022}, thus
the first two lemmas below
follow from \cite[Lemma~5.1]{TJ-KK-KS-2022}.
We note that the estimates \eqref{approx:p} were established in \cite[Theorem 4.1]{MR2917772}, for more general approach see also \cite{MR3666815}.

\begin{lemma}
Let $R\geq 1$. There is a constant $c$ such that for all $|x-z|\geq 1$ and $|w-z|\leq 2R$,
\begin{align}\label{ineq:nu-shift-2}
\nu(x-z)\leq c\, \nu(x-w)\,,
\end{align}
and for all $|x|,|y|\geq 2R$ and $|z|\leq R$,
\begin{align}\label{ineq:nu-prod}
\nu(x-z)\nu(z-y)\leq c\, \nu(x-y)\,.
\end{align}
\end{lemma}

\begin{lemma}
Let $T,R>0$. For all $t\in (0,T]$ and $x,y\in\Rd$,
\begin{align}\label{approx:p}
p(t,x,y)\approx 
\min\{ t^{-d/\alpha}, t\nu(x-y)\}\,.
\end{align}
For all $t\in(0,T]$ and $|x-y|\leq R$,
\begin{align}\label{approx:p-stable}
p(t,x,y)\approx 
\min\{ t^{-d/\alpha}, t|x-y|^{-d-\alpha}\}\,.
\end{align}
For all $t\in (0,T]$ and $|x-y|\geq R$,
\begin{align}\label{approx:p-tnu}
p(t,x,y)\approx t\nu(x-y)\,.
\end{align}
For every $a>0$ there is $c>0$ such that for all $t\in (0,T]$, $z,y\in\Rd$ and $|w|\leq a t^{1/\alpha}$,
\begin{align}\label{ineq:p-shift}
p(t,z+w,y)\leq c p(t,z,y)\,.
\end{align}
\end{lemma}

\begin{lemma}\label{lam:p-deriv-bound}
There is a constant $c$ such that for all $t\geq 1$ and $x,y\in\Rd$
$$
|\partial_t p(t,x,y)| \leq c e^{-t}\,.
$$
For every $x\neq y$ there is a constant $c$ such that for all $t\in (0,1]$,
$$
|\partial_t p(t,x,y)| \leq c\,.
$$
\end{lemma}
\begin{proof}
Note that $p(t,x,y)= e^{-t} (2\pi)^{-d} \int_{\Rd} e^{-i \left<y-x,z\right>}\, e^{-t \Psi(z)} \,\d{z}$,
where $\Psi(x)=(|x|^2+1)^{\alpha/2}-1$.
Both, the above integral and its derivative in $t$, are bounded for $t\geq 1$, proving the first statement.
Now, notice that
$\Psi(x)=(|x|^2+1)^{\alpha/2}-1\approx |x|^2\land |x|^{\alpha}$.
This allows us to use
\cite[Propositions~C.1 and~C.5]{JM-KS-2021}, which gives
$|\partial_t  \, e^{t}p(t,x,y)|\leq c t^{-1}\Upsilon_t(x-y)\leq c$,
see \cite[Section~3]{JM-KS-2021} for the definition of $\Upsilon_t$.
This proves the second part.

\end{proof}

\begin{lemma}\label{lem:toLm}
For all $x\neq y$ we have
$
\lim_{t\to 0^+} \dfrac{p(t,x,y)}{t} = \nu(x-y)
$.
\end{lemma}
\begin{proof}
By \cite[Corollary~8.9]{MR1739520} we have
$\int_0^{\infty}f(s)\frac{\eta_t(s)}{t}\d{s} \longrightarrow \int_0^{\infty}f(s)\Pi(\d{s})$
as $t\to 0^+$
for any bounded continuous $f$ vanishing in the neighbourhood of zero. 
See Subsection~\ref{subsec:rel-op} for the definition of $\Pi(\d{s})$.
It is then not hard to verify that together with \eqref{def:p^m}, for $x\neq y$ we have
\begin{align*}
\frac{p(t,x,y)}{t}=\int_0^{\infty} g_s(x-y)e^{-s}\frac{\eta_t(s)}{t} \,\d{s}\quad \xrightarrow{t\to 0^+} \quad 
 \int_0^\infty  g_s(x-y)e^{-s}\, \Pi(\d{s})
=\nu(x-y)\,.
\end{align*}
\end{proof}

\section*{Declarations}

\subsection*{Funding}
Tomasz Jakubowski and Kamil Kaleta were partially supported by the grant 2015/18/E/ST1/00239 of National Science Centre, Poland.\\
Karol Szczypkowski was partially supported through the DFG-NCN Beethoven Classic~3 programme,
contract no. 2018/31/G/ST1/02252 (National Science Centre, Poland) and SCHI-419/11{1 (DFG, Germany).}

\subsection*{Competing interests}
The authors have no competing interests to declare that are relevant to the content of this article.\\

Data sharing not applicable to this article as no datasets were generated or analysed during the current study.

\small
\bibliographystyle{abbrv}

\begin{thebibliography}{10}

\bibitem{MR2114706}
Erratum to: ``{G}lobal heat kernel bounds via desingularizing weights'' [{J}.
  {F}unct. {A}nal. {\bf 212} (2004), no. 2, 373--398; mr2064932] by {P}. {D}.
  {M}ilman and {Y}u. {A}. {S}emenov.
\newblock {\em J. Funct. Anal.}, 220(1):238--239, 2005.

\bibitem{MR3479207}
B.~Abdellaoui, M.~Medina, I.~Peral, and A.~Primo.
\newblock The effect of the {H}ardy potential in some {C}alder\'{o}n-{Z}ygmund
  properties for the fractional {L}aplacian.
\newblock {\em J. Differential Equations}, 260(11):8160--8206, 2016.

\bibitem{MR3492734}
B.~Abdellaoui, M.~Medina, I.~Peral, and A.~Primo.
\newblock Optimal results for the fractional heat equation involving the
  {H}ardy potential.
\newblock {\em Nonlinear Anal.}, 140:166--207, 2016.

\bibitem{MR4343747}
V.~Ambrosio.
\newblock On the fractional relativistic {S}chr\"{o}dinger operator.
\newblock {\em J. Differential Equations}, 308:327--368, 2022.

\bibitem{MR4381659}
G.~Ascione and J.~L\H{o}rinczi.
\newblock Potentials for non-local {S}chr\"{o}dinger operators with zero
  eigenvalues.
\newblock {\em J. Differential Equations}, 317:264--364, 2022.

\bibitem{MR4206411}
H.~Balsam and K.~Pietruska-Pa{\l}uba.
\newblock Transition density estimates for relativistic {$\alpha$}-stable
  processes on metric spaces.
\newblock {\em Probab. Math. Statist.}, 40(2):183--204, 2020.

\bibitem{MR742415}
P.~Baras and J.~A. Goldstein.
\newblock The heat equation with a singular potential.
\newblock {\em Trans. Amer. Math. Soc.}, 284(1):121--139, 1984.

\bibitem{MR4216547}
A.~BenAmor.
\newblock The heat equation for the {D}irichlet fractional {L}aplacian with
  {H}ardy's potentials: properties of minimal solutions and blow-up.
\newblock {\em Rev. Roumaine Math. Pures Appl.}, 66(1):43--66, 2021.

\bibitem{MR4104934}
M.~Bhakta, A.~Biswas, D.~Ganguly, and L.~Montoro.
\newblock Integral representation of solutions using {G}reen function for
  fractional {H}ardy equations.
\newblock {\em J. Differential Equations}, 269(7):5573--5594, 2020.

\bibitem{MR3514392}
K.~Bogdan, Y.~Butko, and K.~Szczypkowski.
\newblock Majorization, 4{G} theorem and {S}chr\"{o}dinger perturbations.
\newblock {\em J. Evol. Equ.}, 16(2):241--260, 2016.

\bibitem{MR3460023}
K.~Bogdan, B.~Dyda, and P.~Kim.
\newblock Hardy inequalities and non-explosion results for semigroups.
\newblock {\em Potential Anal.}, 44(2):229--247, 2016.

\bibitem{MR3933622}
K.~Bogdan, T.~Grzywny, T.~Jakubowski, and D.~Pilarczyk.
\newblock Fractional {L}aplacian with {H}ardy potential.
\newblock {\em Comm. Partial Differential Equations}, 44(1):20--50, 2019.

\bibitem{MR2457489}
K.~Bogdan, W.~Hansen, and T.~Jakubowski.
\newblock Time-dependent {S}chr\"{o}dinger perturbations of transition
  densities.
\newblock {\em Studia Math.}, 189(3):235--254, 2008.

\bibitem{MR2283957}
K.~Bogdan and T.~Jakubowski.
\newblock Estimates of heat kernel of fractional {L}aplacian perturbed by
  gradient operators.
\newblock {\em Comm. Math. Phys.}, 271(1):179--198, 2007.

\bibitem{KB-TJ-PK-DP-2022}
K.~Bogdan, T.~Jakubowski, P.~Kim, and D.~Pilarczyk.
\newblock {Self-similar solution for {H}ardy operator}.
\newblock preprint 2022, arXiv:2203.02039.

\bibitem{MR3000465}
K.~Bogdan, T.~Jakubowski, and S.~Sydor.
\newblock Estimates of perturbation series for kernels.
\newblock {\em J. Evol. Equ.}, 12(4):973--984, 2012.

\bibitem{MR3156646}
B.~B\"{o}ttcher, R.~Schilling, and J.~Wang.
\newblock {\em L\'{e}vy matters. {III}}, volume 2099 of {\em Lecture Notes in
  Mathematics}.
\newblock Springer, Cham, 2013.
\newblock L\'{e}vy-type processes: construction, approximation and sample path
  properties, With a short biography of Paul L\'{e}vy by Jean Jacod, L\'{e}vy
  Matters.

\bibitem{MR4406449}
H.~Bueno, G.~G. Mamani, A.~H.~S. Medeiros, and G.~A. Pereira.
\newblock Results on a strongly coupled, asymptotically linear
  pseudo-relativistic {S}chr\"{o}dinger system: ground state, radial symmetry
  and {H}\"{o}lder regularity.
\newblock {\em Nonlinear Anal.}, 221:Paper No. 112916, 22, 2022.

\bibitem{MR4319351}
T.~A. Bui and T.~Q. Bui.
\newblock Maximal regularity of parabolic equations associated to generalized
  {H}ardy operators in weighted mixed-norm spaces.
\newblock {\em J. Differential Equations}, 303:547--574, 2021.

\bibitem{AB-PDA-2022}
T.~A. Bui and P.~D'Ancona.
\newblock Generalized {H}ardy operators.
\newblock {\em Nonlinearity}, 36(1):171--198, 2023.

\bibitem{MR4423543}
T.~A. Bui and G.~Nader.
\newblock Hardy spaces associated to generalized {H}ardy operators and
  applications.
\newblock {\em NoDEA Nonlinear Differential Equations Appl.}, 29(4):Paper No.
  40, 2022.

\bibitem{MR1054115}
R.~Carmona, W.~C. Masters, and B.~Simon.
\newblock Relativistic {S}chr\"{o}dinger operators: asymptotic behavior of the
  eigenfunctions.
\newblock {\em J. Funct. Anal.}, 91(1):117--142, 1990.

\bibitem{MR4315592}
H.~Chen and T.~Weth.
\newblock The {P}oisson problem for the fractional {H}ardy operator:
  distributional identities and singular solutions.
\newblock {\em Trans. Amer. Math. Soc.}, 374(10):6881--6925, 2021.

\bibitem{MR2923420}
Z.-Q. Chen, P.~Kim, and R.~Song.
\newblock Global heat kernel estimate for relativistic stable processes in
  exterior open sets.
\newblock {\em J. Funct. Anal.}, 263(2):448--475, 2012.

\bibitem{MR2917772}
Z.-Q. Chen, P.~Kim, and R.~Song.
\newblock Sharp heat kernel estimates for relativistic stable processes in open
  sets.
\newblock {\em Ann. Probab.}, 40(1):213--244, 2012.

\bibitem{MR4163128}
S.~Cho, P.~Kim, R.~Song, and Z.~Vondra\v{c}ek.
\newblock Factorization and estimates of {D}irichlet heat kernels for non-local
  operators with critical killings.
\newblock {\em J. Math. Pures Appl. (9)}, 143:208--256, 2020.

\bibitem{MR4372781}
J.~W. Cholewa and A.~Rodriguez-Bernal.
\newblock On some {PDE}s involving homogeneous operators. {S}pectral analysis,
  semigroups and {H}ardy inequalities.
\newblock {\em J. Differential Equations}, 315:1--56, 2022.

\bibitem{MR763750}
I.~Daubechies.
\newblock One-electron molecules with relativistic kinetic energy: properties
  of the discrete spectrum.
\newblock {\em Comm. Math. Phys.}, 94(4):523--535, 1984.

\bibitem{MR719430}
I.~Daubechies and E.~H. Lieb.
\newblock One-electron relativistic molecules with {C}oulomb interaction.
\newblock {\em Comm. Math. Phys.}, 90(4):497--510, 1983.

\bibitem{MR3237776}
M.~M. Fall and V.~Felli.
\newblock Sharp essential self-adjointness of relativistic {S}chr\"{o}dinger
  operators with a singular potential.
\newblock {\em J. Funct. Anal.}, 267(6):1851--1877, 2014.

\bibitem{MR3393257}
M.~M. Fall and V.~Felli.
\newblock Unique continuation properties for relativistic {S}chr\"{o}dinger
  operators with a singular potential.
\newblock {\em Discrete Contin. Dyn. Syst.}, 35(12):5827--5867, 2015.

\bibitem{MR864658}
C.~Fefferman and R.~de~la Llave.
\newblock Relativistic stability of matter. {I}.
\newblock {\em Rev. Mat. Iberoamericana}, 2(1-2):119--213, 1986.

\bibitem{MR2308757}
S.~Filippas, L.~Moschini, and A.~Tertikas.
\newblock Sharp two-sided heat kernel estimates for critical {S}chr\"{o}dinger
  operators on bounded domains.
\newblock {\em Comm. Math. Phys.}, 273(1):237--281, 2007.

\bibitem{MR2335782}
R.~L. Frank, E.~H. Lieb, and R.~Seiringer.
\newblock Stability of relativistic matter with magnetic fields for nuclear
  charges up to the critical value.
\newblock {\em Comm. Math. Phys.}, 275(2):479--489, 2007.

\bibitem{MR2425175}
R.~L. Frank, E.~H. Lieb, and R.~Seiringer.
\newblock Hardy-{L}ieb-{T}hirring inequalities for fractional {S}chr\"{o}dinger
  operators.
\newblock {\em J. Amer. Math. Soc.}, 21(4):925--950, 2008.

\bibitem{MR4206613}
R.~L. Frank, K.~Merz, and H.~Siedentop.
\newblock Equivalence of {S}obolev norms involving generalized {H}ardy
  operators.
\newblock {\em Int. Math. Res. Not. IMRN}, (3):2284--2303, 2021.

\bibitem{MR4196148}
R.~L. Frank, K.~Merz, H.~Siedentop, and B.~Simon.
\newblock Proof of the strong {S}cott conjecture for {C}handrasekhar atoms.
\newblock {\em Pure Appl. Funct. Anal.}, 5(6):1319--1356, 2020.

\bibitem{MR2778606}
M.~Fukushima, Y.~Oshima, and M.~Takeda.
\newblock {\em Dirichlet forms and symmetric {M}arkov processes}, volume~19 of
  {\em De Gruyter Studies in Mathematics}.
\newblock Walter de Gruyter \& Co., Berlin, extended edition, 2011.

\bibitem{TG-KK-PS-2022}
T.~Grzywny, K.~Kaleta, and P.~Sztonyk.
\newblock Heat kernels of non-local {S}chr\"{o}dinger operators with {K}ato
  potentials.
\newblock {\em J. Differential Equations}, 340:273--308, 2022.

\bibitem{MR2386098}
T.~Grzywny and M.~Ryznar.
\newblock Two-sided optimal bounds for {G}reen functions of half-spaces for
  relativistic {$\alpha$}-stable process.
\newblock {\em Potential Anal.}, 28(3):201--239, 2008.

\bibitem{MR436854}
I.~W. Herbst.
\newblock Spectral theory of the operator {$(p^{2}+m^{2})^{1/2}-Ze^{2}/r$}.
\newblock {\em Comm. Math. Phys.}, 53(3):285--294, 1977.

\bibitem{MR486686}
M.~E.~H. Ismail and M.~E. Muldoon.
\newblock Monotonicity of the zeros of a cross-product of {B}essel functions.
\newblock {\em SIAM J. Math. Anal.}, 9(4):759--767, 1978.

\bibitem{TJ-KK-KS-2022}
T.~Jakubowski, K.~Kaleta, and K.~Szczypkowski.
\newblock Bound states and heat kernels for fractional-type {S}chr\"odinger
  operators with singular potentials.
\newblock preprint 2022, arXiv:2208.00687.

\bibitem{TJ-PM-2022}
T.~Jakubowski and P.~Maciocha.
\newblock {Ground-state representation for fractional {L}aplacian on half-line
  }.
\newblock preprint 2022, arXiv:2206.05157.

\bibitem{MR4140086}
T.~Jakubowski and J.~Wang.
\newblock Heat kernel estimates of fractional {S}chr\"{o}dinger operators with
  negative {H}ardy potential.
\newblock {\em Potential Anal.}, 53(3):997--1024, 2020.

\bibitem{MR4100844}
K.~Kaleta and R.~L. Schilling.
\newblock Progressive intrinsic ultracontractivity and heat kernel estimates
  for non-local {S}chr\"{o}dinger operators.
\newblock {\em J. Funct. Anal.}, 279(6):108606, 69, 2020.

\bibitem{MR3666815}
K.~Kaleta and P.~Sztonyk.
\newblock Small-time sharp bounds for kernels of convolution semigroups.
\newblock {\em J. Anal. Math.}, 132:355--394, 2017.

\bibitem{MR2316878}
P.~Kim and Y.-R. Lee.
\newblock Generalized 3{G} theorem and application to relativistic stable
  process on non-smooth open sets.
\newblock {\em J. Funct. Anal.}, 246(1):113--143, 2007.

\bibitem{MR2231884}
T.~Kulczycki and B.~Siudeja.
\newblock Intrinsic ultracontractivity of the {F}eynman-{K}ac semigroup for
  relativistic stable processes.
\newblock {\em Trans. Amer. Math. Soc.}, 358(11):5025--5057, 2006.

\bibitem{MR2583992}
E.~H. Lieb and R.~Seiringer.
\newblock {\em The stability of matter in quantum mechanics}.
\newblock Cambridge University Press, Cambridge, 2010.

\bibitem{MR1953267}
V.~Liskevich and Z.~Sobol.
\newblock Estimates of integral kernels for semigroups associated with
  second-order elliptic operators with singular coefficients.
\newblock {\em Potential Anal.}, 18(4):359--390, 2003.

\bibitem{MR4311597}
K.~Merz.
\newblock On scales of {S}obolev spaces associated to generalized {H}ardy
  operators.
\newblock {\em Math. Z.}, 299(1-2):101--121, 2021.

\bibitem{MR4375654}
G.~Metafune, L.~Negro, and C.~Spina.
\newblock {$L^p$} estimates for the {C}affarelli-{S}ilvestre extension
  operators.
\newblock {\em J. Differential Equations}, 316:290--345, 2022.

\bibitem{MR3630331}
G.~Metafune, M.~Sobajima, and C.~Spina.
\newblock Kernel estimates for elliptic operators with second-order
  discontinuous coefficients.
\newblock {\em J. Evol. Equ.}, 17(1):485--522, 2017.

\bibitem{MR2064932}
P.~D. Milman and Y.~A. Semenov.
\newblock Global heat kernel bounds via desingularizing weights.
\newblock {\em J. Funct. Anal.}, 212(2):373--398, 2004.

\bibitem{JM-KS-2021}
J.~Minecki and K.~Szczypkowski.
\newblock {Non-symmetric L{\'e}vy-type operators}.
\newblock preprint 2021, arXiv:2112.13101.

\bibitem{MR2328115}
L.~Moschini and A.~Tesei.
\newblock Parabolic {H}arnack inequality for the heat equation with
  inverse-square potential.
\newblock {\em Forum Math.}, 19(3):407--427, 2007.

\bibitem{MR3020137}
D.~Pilarczyk.
\newblock Self-similar asymptotics of solutions to heat equation with inverse
  square potential.
\newblock {\em J. Evol. Equ.}, 13(1):69--87, 2013.

\bibitem{MR4294647}
L.~Roncal.
\newblock Hardy type inequalities for the fractional relativistic operator.
\newblock {\em Math. Eng.}, 4(3):Paper No. 018, 16, 2022.

\bibitem{MR1906405}
M.~Ryznar.
\newblock Estimates of {G}reen function for relativistic {$\alpha$}-stable
  process.
\newblock {\em Potential Anal.}, 17(1):1--23, 2002.

\bibitem{MR1739520}
K.-I. Sato.
\newblock {\em L\'evy processes and infinitely divisible distributions},
  volume~68 of {\em Cambridge Studies in Advanced Mathematics}.
\newblock Cambridge University Press, Cambridge, 1999.
\newblock Translated from the 1990 Japanese original, Revised by the author.

\bibitem{MR2978140}
R.~L. Schilling, R.~Song, and Z.~Vondra\v{c}ek.
\newblock {\em Bernstein functions}, volume~37 of {\em De Gruyter Studies in
  Mathematics}.
\newblock Walter de Gruyter \& Co., Berlin, second edition, 2012.
\newblock Theory and applications.

\bibitem{MR4020755}
S.~Secchi.
\newblock A generalized pseudorelativistic {S}chr\"{o}dinger equation with
  supercritical growth.
\newblock {\em Commun. Contemp. Math.}, 21(8):1850073, 21, 2019.

\bibitem{MR2925182}
N.-R. Shieh.
\newblock On time-fractional relativistic diffusion equations.
\newblock {\em J. Pseudo-Differ. Oper. Appl.}, 3(2):229--237, 2012.

\bibitem{RS-PW-SW-2022}
R.~Song, P.~Wu, and S.~Wu.
\newblock {Heat kernel estimates for non-local operator with multisingular
  critical killing}.
\newblock preprint 2022, arXiv:2203.03891.

\bibitem{MR4116699}
K.~Tsuchida.
\newblock On a construction of harmonic function for recurrent relativistic
  {$\alpha$}-stable processes.
\newblock {\em Tohoku Math. J. (2)}, 72(2):299--315, 2020.

\bibitem{MR507514}
H.~van Haeringen.
\newblock Bound states for {$r{}^{-}{}^{2}$}-like potentials in one and three
  dimensions.
\newblock {\em J. Math. Phys.}, 19(10):2171--2179, 1978.

\bibitem{MR1760280}
J.~L. Vazquez and E.~Zuazua.
\newblock The {H}ardy inequality and the asymptotic behaviour of the heat
  equation with an inverse-square potential.
\newblock {\em J. Funct. Anal.}, 173(1):103--153, 2000.

\bibitem{MR0010746}
G.~N. Watson.
\newblock {\em A {T}reatise on the {T}heory of {B}essel {F}unctions}.
\newblock Cambridge University Press, Cambridge, England; The Macmillan
  Company, New York, 1944.

\bibitem{MR351332}
R.~A. Weder.
\newblock Spectral properties of one-body relativistic spin-zero
  {H}amiltonians.
\newblock {\em Ann. Inst. H. Poincar\'{e} Sect. A (N.S.)}, 20:211--220, 1974.

\bibitem{MR0402547}
R.~A. Weder.
\newblock Spectral analysis of pseudodifferential operators.
\newblock {\em J. Functional Analysis}, 20(4):319--337, 1975.

\bibitem{MR3637943}
Z.-H. Yang and S.-Z. Zheng.
\newblock The monotonicity and convexity for the ratios of modified {B}essel
  functions of the second kind and applications.
\newblock {\em Proc. Amer. Math. Soc.}, 145(7):2943--2958, 2017.

\end{thebibliography}

\end{document}